\RequirePackage{fix-cm}
\documentclass[smallextended]{svjour3}       % onecolumn (second format)
\smartqed  % flush right qed marks, e.g. at end of proof
\usepackage{graphicx}
\usepackage{amsmath,amssymb, amsfonts}
\usepackage{enumitem}
\usepackage{tikz}
\usepackage[linkcolor=blue,colorlinks=true,citecolor=blue]{hyperref}%,epsfig, color
\newtheorem{thm}{Theorem}[section]

\newtheorem{lem}[thm]{Lemma}
\newtheorem{defn}[thm]{Definition}
\newtheorem{prop}[thm]{Proposition}
\newtheorem{coro}[thm]{Corollary}
\newtheorem{rmk}[thm]{Remark}

\numberwithin{equation}{section}

\smartqed
\newcommand{\N}{\mathbb{N}}
\newcommand{\C}{\mathbb{C}}
\newcommand{\R}{\mathbb{R}}
\newcommand{\abs}[1]{\left| {#1} \right|}
\newcommand{\bd}{\partial}
\newcommand{\dist}{\operatorname{\mathrm{dist}}}
\newcommand{\norm}[1]{\left\lVert{#1}\right\rVert}
\renewcommand{\hat}[1]{\widehat{#1}}
\renewcommand{\tilde}{\widetilde}
\renewcommand{\chi}{\mathbf{1}}

\journalname{Geometriae Dedicata}

\begin{document}

\title{Shrinking rates of horizontal gaps for generic translation surfaces \thanks{The first author is supported by NSF grants DMS-2055354 and DMS-452762, the Sloan foundation, Poincar{\'e} chair, and Warnock chair. The second author is supported by the Deutsche Forschungsgemeinschaft (DFG) -- Projektnummer 445466444.}
}

%\subtitle{Do you have a subtitle?\\ If so, write it here}

%\titlerunning{Short form of title}        % if too long for running head

\author{Jon Chaika         \and
        Samantha Fairchild %etc.
}

%\authorrunning{Short form of author list} % if too long for running head

\institute{J. Chaika \at
              University of Utah, Salt Lake City, Utah, USA \\
              orcid: 000-0002-9577-0109\\
              \email{chaika@math.utah.edu}           %  \\
%             \emph{Present address:} of F. Author  %  if needed
           \and
           S. Fairchild (Corresponding Author) \at
              Max Planck Institute, Leipzig Germany\\
              orcid: 0000-0001-7265-6904\\
              \email{samantha.fairchild@mis.mpg.de}
}

\date{Received: date / Accepted: date}
% The correct dates will be entered by the editor

\maketitle

\begin{abstract}
A translation surface is given by polygons in the plane, with sides identified by translations to create a closed Riemann surface with a flat structure away from finitely many singular points. Understanding geodesic flow on a surface involves understanding saddle connections. Saddle connections are the geodesics starting and ending at these singular points and are associated to a discrete subset of the plane. To measure the behavior of saddle connections of length at most $R$, we obtain precise decay rates as $R\to \infty$ for the difference in angle between two almost horizontal saddle connections.

\keywords{Translation Surfaces \and Dynamical Systems \and Saddle Connections \and Mixing}
% \PACS{PACS code1 \and PACS code2 \and more}
\subclass{37E35 \and 32G15 \and 37A25}
\vspace{.05in}
\noindent
\textbf{Statements and Declarations:}
\begin{enumerate}
\item Funding: The first author is supported by NSF grants DMS-2055354 and DMS-452762, the Sloan foundation, Poincar{\'e} chair, and Warnock chair. The second author was partially supported by the Deutsche Forschungsgemeinschaft (DFG) -- Projektnummer 445466444 and 507303619.
\item Author Contributions: All authors wrote and reviewed the manuscript
\item Data Availability Statement: Data sharing not applicable to this article as no datasets were generated or analysed during the current study.
\end{enumerate}
\end{abstract}

%\section{Introduction}
%\label{intro}
%Your text comes here. Separate text sections with
%\section{Section title}
%\label{sec:1}
%Text with citations \cite{RefB} and \cite{RefJ}.
%\subsection{Subsection title}
%\label{sec:2}
%as required. Don't forget to give each section
%and subsection a unique label (see Sect.~\ref{sec:1}).
%\paragraph{Paragraph headings} Use paragraph headings as needed.
%\begin{equation}
%a^2+b^2=c^2
%\end{equation}

% For one-column wide figures use
%\begin{figure}
% Use the relevant command to insert your figure file.
% For example, with the graphicx package use
%  \includegraphics{example.eps}
% figure caption is below the figure
%\caption{Please write your figure caption here}
%\label{fig:1}       % Give a unique label
%\end{figure}
%
% For two-column wide figures use
%\begin{figure*}
% Use the relevant command to insert your figure file.
% For example, with the graphicx package use
%  \includegraphics[width=0.75\textwidth]{example.eps}
% figure caption is below the figure
%\caption{Please write your figure caption here}
%\label{fig:2}       % Give a unique label
%\end{figure*}
%
% For tables use
%\begin{table}
% table caption is above the table
%\caption{Please write your table caption here}
%\label{tab:1}       % Give a unique label
% For LaTeX tables use
%\begin{tabular}{lll}
%\hline\noalign{\smallskip}
%first & second & third  \\
%\noalign{\smallskip}\hline\noalign{\smallskip}
%number & number & number \\
%number & number & number \\
%\noalign{\smallskip}\hline
%\end{tabular}
%\end{table}
\section{Introduction}
Consider a finite collection of polygons in the plane, where all sides come in pairs of equal length with opposite orientations on the boundary of the polygons. Identifying these sides gives a compact finite type Riemann surface. The form $dz$ on the plane endows this Riemann surface with a holomorphic 1-form $\omega$. This structure is called a \textit{translation surface}. {Translation surfaces are stratified by their zeros and each connected component of each stratum supports a natural probability measure supported on unit area surfaces called \emph{Masur--Smillie--Veech (MSV)} measure. More background on translation surfaces can be found, for example in \cite{Massart22,MasurSmillie91}.} %Translation surfaces are stratified by their zeros and each connected component of each stratum supports a natural measure class called \emph{Masur--Smillie--Veech (MSV)} measure. More background on translation surfaces can be found, for example in \cite{Massart22,MasurSmillie91}.
A \textit{saddle connection} $\gamma$ is a geodesic starting and ending at the zeroes of $\omega$ with no zeroes in between. Associated to $\gamma$, we define the \textit{holonomy vector} $v_{\gamma} = \int_{\gamma} \,d\omega \in \C$. To understand the geometry of a typical surface, much effort has gone into understanding the asymptotic behavior of holonomy vectors of saddle connections of length at most $R$ \cite{Dozier19,EskinMasur01,EMZ03,EMM15,Masur88Lower,Masur90QuadraticGrowth,NRW,Veech98,Vorobets05}, and more recently the asymptotic behavior of pairs of saddle connections \cite{ACM19,AFM22}. We will consider
$$\Lambda_\omega(R)= \{v_\gamma \in \C \cap B(0,R): \gamma \text{ is a saddle connection}\}$$
and
$$ \Theta_\omega(R) = \{\mbox{arg}(v_\gamma): v_\gamma \in \Lambda_{\omega}(R)\},$$
where $\mbox{arg}(v) \in [-\pi,\pi)$ is the angle $v$ makes with the horizontal. The sets $\Lambda_\omega(R)$ and $\Theta_\omega(R)$ are discrete subsets of $\C$ and $[-\pi, \pi)$, respectively, so we can define $$\zeta_\omega(R) = \min\{\phi \in \Theta_{\omega}(R): \phi \geq 0\} - \max\{\phi \in \Theta_{\omega}(R) : \phi < 0\}.$$
The main result of this paper is:

 \begin{thm}\label{thm:01law} 
	Let $\psi: [1,\infty) \to [1,\infty)$ be a nondecreasing function. In any connected component of a stratum of translation surfaces of genus at least 2, 
	\begin{enumerate}
		\item If $\int_1^\infty \frac{1}{t\psi(t)^2}\,dt < \infty$, then for MSV almost every $\omega$,
		$$\liminf_{R\to\infty}  \psi(R) R^2\zeta_\omega(R) = \infty.$$
		\item If $\int_1^\infty \frac{1}{t\psi(t)^2} \,dt = \infty$, then for MSV almost every $\omega$, 
		$$\liminf_{R\to\infty} \psi(R) R^2 \zeta_\omega(R) = 0.$$
	\end{enumerate}
\end{thm}

\begin{rmk}
Note that the choice of a horizontal gap is a convenience. Apply a rotation to the full measure set in Theorem~\ref{thm:01law}, and we obtain the same result in a different direction. Consider a countable subset $D_n= \{\theta_n\}_{n\in\mathbb{N}} \subseteq [0,2\pi)$. Since a countable union of measure 0 subsets is still measure zero, we obtain a natural corollary that the smallest gap of any of the directions in $D_n$ has the same decay rate as given in Theorem~\ref{thm:01law}.
\end{rmk}
 
We first provide an explanation for the scaling factor of $R^2$. Masur (\cite{Masur88Lower,Masur90QuadraticGrowth}) showed $|\Lambda_\omega(R)|$ has \emph{quadratic growth} in the sense that for each $\omega$ there exist constants $c_1,c_2$ so that 
$$c_1R^2\leq |\Lambda_\omega(R)|\leq c_2R^2$$ for all large enough $R$. {Because the total angle about the singular points for the flat metric is at most $4\pi(2g-2)$ and every saddle connection begins at ends at a singular point,} there are at most $4g-4$ saddle connections in the same direction, so $|\Theta_\omega(R)|$ also has quadratic growth. This result explains the scaling factor of $R^2$ in Theorem~\ref{thm:01law}. The quadratic growth of saddle connections was subsequently built on in \cite{EskinMasur01,EMM15,NRW,Veech98,Vorobets05}.

For every translation surface, \cite{Masur86} shows that $\Theta_{\omega} = \bigcup_R  \Theta_{\omega}(R)$ is dense in $[-\pi,\pi)$.  If we order the points $\theta_1 \leq \cdots \leq \theta_{|\Theta_\omega(R)|}$ in $\Theta_\omega(R)$, then by density the adjacent differences $\theta_{j+1} -\theta_j \to 0$ as $R\to\infty$ for every $\omega$. Thus $\zeta_\omega(R) \to 0$ as $R\to\infty$ for every $\omega$. If we reduce to almost every translation surface, Theorem~\ref{thm:01law} give a rate of convergence for a single direction, yielding partial information on how $\Theta_\omega(R)$ is distributed in $[-\pi, \pi)$. For example, by considering $\psi(t) = \log(t)$ and $\psi(t) = \sqrt{t}$, we know for a typical surface that the rate of convergence is shrinking faster than $(\log(R)R^2)^{-1}$ {along a subsequence} and slower than $R^{-\frac{5}{2}}$. However, we cannot expect Theorem~\ref{thm:01law} to hold for every translation surface. Indeed when $\omega$ is a \emph{lattice surface} (see \cite[Sections 5 and 7]{Massart22} for a definition), \cite{AthreyaChaika12} showed for every unbounded $\psi$ we have $\liminf_{R \to \infty}\psi(R)R^2 \zeta_{\omega}(R)=\infty$.

It is natural question to ask if Theorem~\ref{thm:01law} can be extended to measures supported on $\mathrm{SL}(2,\R)$-orbit closures which are not closed (lattice surfaces) or dense (connected component of a stratum). Though many pieces of our argument have analogous results for orbit closures, the local coordinate computations require great care centered on a very specific surface. In light of this, we save this question for future work.

The results of Theorem~\ref{thm:01law} consider the behavior of a single gap. There is also substantial work done on studying the behavior of the family of gaps. Namely one can study the entire set $\Theta_\omega(R)$. The distribution of normalized gaps exists for almost every $\omega$ by \cite{AthreyaChaika12}. In many cases, the distribution has also been computed \cite{ACL15,gap2ngon,KSW21,Sanchez21,UW16}. Considering the behavior of a single gap, which is the focus of the current paper, is orthogonal because the behavior of a single gap does not affect the distribution of gaps.

\subsection{Outline of proof}
The proof follows the now standard strategy of relating a problem about the geometry of a translation surface $\omega$ to the orbit of $\omega$ under Teichm\"uller geodesic flow, $g_t=\begin{pmatrix}e^t&0\\0&e^{-t}\end{pmatrix}$. In Section~\ref{sec:2} we reduce our problem about gaps to a \emph{shrinking target problem} for $g_t$. The shrinking targets are sets $A_t$ obtained by relating Theorem~\ref{thm:01law} to whether or not $g_t\omega \in A_{t}$ for arbitrarily large $t$. To prove Theorem~\ref{thm:01law} (2) we use independence results for the sets $g_{-t} A_t$.  A key tool to do this is the fact that the $g_t$ action is exponentially mixing (Section~\ref{sec:ExpDecayFar}). However, because our targets are not $\mathrm{SO}(2)$-invariant, the estimates from exponential mixing are not sufficient to treat all non-increasing sequences. See Remark~\ref{rmk:technical}  and Assumptions (3) and (4) in Proposition~\ref{prop:Exp_Decay_Borel_Cantelli}. (c.f.~\cite[Proposition 11.1]{KL21}). Section~\ref{sec:averages} establishes that Assumption (4) of Proposition~\ref{prop:Exp_Decay_Borel_Cantelli} is satisfied and provides a different argument to overcome the limitations from exponential mixing. In fact, in Section~\ref{sec:averages} even the targets we consider are different.  

\section{Reductions}\label{sec:2}

In this section we present a series of reductions: 

\begin{enumerate}
\item In Section~\ref{sec:defmeas} and Section~\ref{sec:targlem} we use renormalization to relate gaps to a shrinking target problem.
\item In Section~\ref{sec:convcase}, we consider the convergence case Theorem~\ref{thm:01law}(1) and prove it suffices to show that for almost every $\omega$ we have 
$$\liminf_{R\to\infty}  \psi(R) R^2\zeta_\omega(R) >0.$$
\item In Section~\ref{sec:divcase}, we consider the divergence case Theorem~\ref{thm:01law}(2) and prove it suffices to show that for a positive measure set of $\omega$ we have $$\liminf_{R\to\infty}  \psi(R) R^2\zeta_\omega(R)<\infty.$$
\item In Section~\ref{sec:axiom} we state and prove a partial converse to the Borel--Cantelli lemma (Proposition~\ref{prop:Exp_Decay_Borel_Cantelli}) that we use as the framework for proving Theorem~\ref{thm:01law}. 
\item {We conclude in Section~\ref{sec:proof01law} by proving Theorem~\ref{thm:01law} under one additional assumption. This additional assumption is in Proposition~\ref{prop:sufficient is satisfied}, which states the existence of the sets needed for Proposition~\ref{prop:Exp_Decay_Borel_Cantelli}.}

\end{enumerate}
{Section~\ref{sec:ExpDecayFar} and Section~\ref{sec:averages} are devoted to proving Proposition~\ref{prop:sufficient is satisfied}.}

\subsection{Definition and measure of the shrinking target sets} \label{sec:defmeas}

Fix $\mathcal{H}$ a connected component of a stratum of translation surfaces, {unmarked in the sense of \cite[Section 2.3]{CSW20}}. Note $\mathcal{H}$ has complex dimension $2g + s - 1$ with $s$ the number of distinct singularities. Fix $0<\delta<1$ as given by exponential mixing of the geodesic flow (Theorem~\ref{thm:AGY}), which depends only on $\mathcal{H}$. We will sample $\psi$ along a discrete set $\psi(b^k)$ for $k \in \N$ where $b = e^{\ell_0}$, and $\ell_0 \geq 1$ is as in Corollary~\ref{cor:AverageUpperBound}, for the given choice of $\delta$ and $I = (-\frac{\pi}{12}, \frac{\pi}{12})$. The definition of $\ell_0$ in turn depends on other constants in Section~\ref{sec:circleaverages}, which is a self-contained section with no dependency on the previous statements.

{We now define the primary sets that we use to develop our shrinking target problem.}
\begin{defn}[Definition of the A's]\label{def:As}
Fix $0<\sigma < 1$, and let $0\leq c<1$. Define $T_{c,\sigma,j}^\pm = T_{c,\sigma, j,\psi,b}^\pm \subseteq \C$ to be the trapezoids with corners given by $$c, 1, 1\pm i \frac{ \sigma}{\psi(b^j)}, c\pm i c\frac{\sigma}{\psi(b^j)}.$$ Set 
\begin{align*}
&H_{c,\sigma,j} = \left\{ \omega \in \mathcal{H}: \substack{\omega \text{ has a holonomy vector in } T_{c,\sigma,j}^+ \\ \text{ and a holonomy vector in  } T_{c,\sigma,j}^-}\right\}.\end{align*}
Finally for $k\in \N$, define $A_k = A_k(c,\sigma)=A_k(c,\sigma,\psi,b)=  g_{\log(b^k)} H_{c,\sigma, k}.$
%where $g_t = \begin{pmatrix}	e^t & 0 \\ 0 & e^{-t}\end{pmatrix}.$

\end{defn}
\begin{rmk}
	In the following we drop the dependence on $\psi$ and $b$ as they are fixed whenever we consider these sets. When the choice of $c$ or $\sigma$ is clear, or arbitrary, then for clarity we will suppress the dependence and simply write $A_k$ instead of $A_k(c,\sigma)$.
\end{rmk}

The remainder of this subsection is devoted to obtaining measure bounds for $A_k$. Let $\mu$ denote the MSV probability measure on $\mathcal{H}$, whose support is the locus of unit area surfaces in $\mathcal{H}$. Since $\mu$ is $\mathrm{SL}(,\mathbb{R})$-invariant, it suffices to understand $\mu(H_{c,\sigma,j})$.

\begin{lem}\label{lem:measurebounds}
	 Given a stratum $\mathcal{H}$, there exists positive finite constants $m = m(\mathcal{H})$, $M = M(\mathcal{H})$, %$n_{\mathcal{H}} \in \mathbb{N}$
	 $c_{\mathcal{H}}<1$, and $\sigma_\mathcal{H} <1$  chosen so that for all $0< \sigma < \sigma_\mathcal{H}$ %and for $c_{\mathcal{H}} = 1- 2^{-n_\mathcal{H}}$
	$$\frac{m \sigma^2}{\psi(b^j)^2} \leq \mu(H_{c_\mathcal{H},\sigma,j}) \leq \mu(H_{0,\sigma,j}) \leq \frac{M \sigma^2}{\psi(b^j)^2}.$$
\end{lem}
Before proceeding with the proof, we quote the following result of Masur--Smillie (as quoted from \cite{AthreyaChaika12}), and then develop the idea behind their result before proceeding to the proof of Lemma~\ref{lem:measurebounds}.

\begin{lem} \label{lem:MasurSmillie}
	There is a constant $M$ so that for all $\epsilon, \kappa > 0$, the subset of $\mathcal{H}$ consisting of flat surfaces which have a saddle connection of length at most $\epsilon$ has measure at most $M \epsilon^2$. The subset of flat surfaces which have a saddle connection of length at most $\epsilon$ and another nonhomologous saddle connection of length at most $\kappa$ has measure at most $M \epsilon^2 \kappa^2$. 
\end{lem}

We remark that all uses of $M$ here are not necessarily the same, but vary by at most a multiplicative constant, which we can see for example in the proof of Lemma~\ref{lem:measurebounds}. 

The upper bound of Lemma~\ref{lem:measurebounds} will follow from Lemma~\ref{lem:MasurSmillie}. In order to obtain the lower bound, we first give some background on the construction of the measure $\mu$ following \cite[p.464-5]{MasurSmillie91}. 

First, we allow $\mu$ to be a finite measure for the rest of the subsection, and note that the normalization to a probability measure will only affect the multiplicative constants up to dividing by the measure of the stratum. Now we can define $\mu$ on the flat structures of area $1$ via a cone measure $\tilde{\mu}$ over flat structures with area at most 1. The cone measure $\tilde{\mu}$ is inherited from a measure on relative cohomology, defined via charts coming from the developing map. For $\omega \in \mathcal{H}$ not in a proper orbifold locus, there exists $r_\omega >0$ so that the ball of radius $r_\omega$ about the pre-image of $\omega$ in relative cohomology is sent injectively to the coordinate chart about $\omega$. For a point in a proper orbifold locus of $\mathcal{H}$, we may choose a chart as in \cite[Section 2.3]{CSW20}. Such an orbifold chart has the identification of $\tilde{\mu}$ and Lebesgue measure on relative cohomology on the image of the chart via an almost everywhere $k$-to-1 mapping, where $k$ is the cardinality of the local group of the \emph{orbifold substrata}. In this situation, $r_\omega$ is chosen so that the map is $k$-to-$1$, off of any orbifold substrata intersected with the ball.

{In order to obtain lower bounds on the measure of a set, it suffices to obtain measure bounds on the cone measure $\tilde{\mu}$ in a fixed coordinate chart. To do this, we work in local coordinates by writing} $\omega = (x_1,x_2, x_3) \in \C\times \C \times \C^{2g+s-3}$ and $\widetilde{\mu}$ as Lebesgue measure on $\C^{2g + s-1}$. Consider the Euclidean metric using the notation $|\cdot|$ on each component by identifying $\C^\ell$ with $\R^{2\ell}$. Define $B_{\C^\ell}(z,r) = \{w\in \C^\ell: |w-z| <r\}$. We are now ready to prove Lemma~\ref{lem:measurebounds}.

\begin{theopargself}
\begin{proof}[of Lemma~\ref{lem:measurebounds}]
	For the upper bound, flowing by geodesic flow which preserves measure, $g_{\log\left(\sqrt{\sigma/\psi(b^k)}\right)} H_{0,\sigma,k}$ has two non-homologous vectors of length at most $\sqrt{\frac{2\sigma}{\psi(b^k)}}$, so by Lemma~\ref{lem:MasurSmillie} we obtain the desired upper bound.

	For the lower bound, we work in local coordinates around a surface $w_0$ with two horizontal saddle connections of length $1$. That is, in local coordinates, let $\omega_0= (x_1^0, x_2^0 , x_3^0)$ {be an area $1$ surface}, and let $r>0$ be the injectivity radius, or $k$-to-$1$ radius as described above. One can find such a surface in every connected component of every stratum. Indeed, \cite{KZ03} explicitly constructs surfaces, which are a single horizontal cylinder, in each connected component of every stratum of genus at least two, and then after flowing by $g_t$ if necessary so the cylinder has circumference at least $1$, we can obtain the desired representative. From such a surface one can vary the length of a pair of boundary horizontal saddle connections to obtain $\omega$. As remarked at the end of \cite[Section 1]{ZorichJS}, the two horizontal vectors are in fact non-homologous, and we can pick a basis of homology with the period coordinates represented by saddle connections (see \cite[Appendix A]{Per-Teich} and \cite[Proof of Theorem 4.1]{BG21Systoles}). {To work with} local coordinates, we will first prove the lower bound for $\tilde{\mu}$ and then %use the proof to 
	derive the lower bound for $\mu$.
		
	If necessary, shrink $r$ so that $r<1$ and the image of the chart is contained in a compact subset of $\mathcal{H}$. Notice that for any $c$, since $\psi(t) \geq 1$, $T_{c,\sigma, j}^\pm$ is always contained in the trapezoid $T_{c,\sigma,*}$ with vertices $c \pm ic\sigma, 1\pm i\sigma.$ So we can guarantee $T^\pm_{c,\sigma,j}$ is always contained in the chart whenever 
	 $T_{c,\sigma, *} \subset B_\C(1, r)$. This geometric condition can be satisfied for some $0< \sigma_{\mathcal{H}} < r$, $c_{\mathcal{H}}$ close to 1, and $\sigma < \sigma_{\mathcal{H}}$.  Define the set in $\C\times \C \times \C^{2g + s-3}$ by
\begin{equation} \label{eq:Balldef}
\tilde{H}_{c_\mathcal{H},\sigma,j} = T_{c_\mathcal{H},\sigma, j}^+ \times T_{c_\mathcal{H},\sigma,j}^- \times B \quad\text{ for } B= B_{\C^{2g+s-3}}(x_3^0,r).\end{equation}	By our choice of $r$ the measure  $\widetilde{\mu}$ on $\mathcal{H}$ is locally Lebesgue. Hence if ${{\textbf{m}}}_j$ is Lebesgue measure on $\C^j$, by symmetry of $T^\pm_{c_\mathcal{H},\sigma,j}$, {there is a constant $\tilde{m}$ so that}
	$$\widetilde{\mu}(\tilde{H}_{c_\mathcal{H},\sigma,j}) \geq \tilde{m}\cdot {\textbf{m}}_1(T^+_{c_\mathcal{H},\sigma,j})^2 {\textbf{m}}_{2g+s-3}(B).$$

	We compute the Lebesgue measure of the trapezoids by
	$${{\textbf{m}}}_1(T^\pm_{c_\mathcal{H},\sigma,j}) = \frac{1}{2} \left(\frac{\sigma}{\psi(b^j)} + c_\mathcal{H} \frac{\sigma}{\psi(b^j)}\right) (1-c_\mathcal{H})  = \frac{\sigma}{\psi(b^j)} \frac{(1-c_\mathcal{H}^2)}{2}.$$

The other $2g-3$ coordinates and the possibility that we are in a coordinate patch determined by an orbifold substrata only change these bounds by a multiplicative constant. Thus there are constants $m, m'$ depending on the dimension of $\mathcal{H}$ and $c_\mathcal{H}$ so that 
	\begin{equation}\label{eq:tildest}\tilde{\mu}({{\tilde H}}_{c_\mathcal{H},\sigma,j})  \geq m'\cdot {\tilde{m}\cdot}{\textbf{m}}_1(T^\pm_{c_\mathcal{H},\sigma,j})^2 = m \frac{\sigma^2}{\psi(b^j)^2}.\end{equation}		
	{To complete the proof, we need to pass from $\widetilde{\mu}$ to the cone measure $\mu$. {Specifically recall on a measurable set $X$ of area 1 translations surfaces, $$\mu(X) = \widetilde{\mu}(\{tx: x\in X, \, 0<t<1\}).$$ Notice that scaling each of the coordinates by a number $0<t<1$ decreases the area of the surface. Notice also that the area of the translation surface changes continuously as a function of the fixed local coordinates, and $\omega_0$ has area 1 with coordinates $((1,0),(1,0),x_3^0)$}. Thus, given $r, c_{\mathcal{H}}<1,\sigma_{\mathcal{H}}$ as above, we can choose  $\epsilon>0$, $r'$, $c'_{\mathcal{H}}$ and $\sigma_{\mathcal{H}}'$ satisfying $0<r'<r-{ |x_0|}\epsilon$, $c_{\mathcal{H}}<c_{\mathcal{H}}'<1$ and $0<\sigma_{\mathcal{H}}{'}<\sigma_{\mathcal{H}}$ so that every $\omega'$ in our coordinate patch with local coordinates in $$ T_{c_\mathcal{H}',\sigma, j}^+ \times T_{c_\mathcal{H}',\sigma,j}^- \times B((1-\epsilon)x_3^0,r')$$ with $\sigma<\sigma_{\mathcal{H}}'$ has area less than 1 and arises as $\omega'=t\omega$ for $\omega$ in our coordinate patch, with area 1 and with local coordinates in 
$$ T_{c_\mathcal{H},\sigma, j}^+ \times T_{c_\mathcal{H},\sigma,j}^- \times B.$$  We conclude by observing that the above estimate of $\tilde{\mu}$ {in \eqref{eq:tildest}} works to estimate 
$$\tilde{\mu}\left( T_{c_\mathcal{H}',\sigma, j}^+ \times T_{c_\mathcal{H}',\sigma,j}^- \times B((1-\epsilon)x_3,r') \right) $$ for all $\sigma<\sigma_{\mathcal{H}}'$ so long as $\epsilon>0$ is small enough and $r'$, $c'_{\mathcal{H}}$ and $\sigma_{\mathcal{H}}'$ satisfy $0<r'<r-{ |x_0|}\epsilon$, $c_{\mathcal{H}}<c_{\mathcal{H}}'<1$ and $0<\sigma_{\mathcal{H}}'<\sigma_{\mathcal{H}}$.}
		\qed
	\end{proof}
\end{theopargself}

\subsection{A lemma on targets} \label{sec:targlem}
For the following lemma and corollary, for generality we allow any $b>1$. Note that $b=e^{\ell_0}$ as defined before Definition~\ref{def:As} satisfies $b>1$ since we choose $\ell_0  \geq 1 $.
\begin{lem}\label{lem:Cauchy trick} Let $\phi:[1,\infty)\to [1,\infty)$ be nondecreasing. We have $$\int_1^\infty (t\phi(t))^{-1} dt=\infty$$ if and only if $$\sum_{j=1}^\infty \phi(b^j)^{-1}=\infty$$ for any $b>1$. 
\end{lem}

\begin{proof}
For ease of exposition we assume $b=2$, the general case is similar, but requires using floor and ceiling functions. 
For each $k\in \mathbb{N}$ we have 
$$2^k \frac 1 {2^k\phi(2^k)}\geq\int_{2^k}^{2^{k+1}}(t\phi(t))^{-1}dt\geq 2^k\frac 1{2^{k+1}\phi(2^{k+1})}.$$
It follows that $\sum_{j=1}^\infty \phi(2^j)^{-1} \geq \int_2^\infty (t\phi(t))^{-1}dt\geq \frac{1}{2} \sum_{j=2}^\infty \phi(2^j)^{-1} .$\qed
\end{proof}

\begin{coro}\label{cor:psidivergence} Let $\psi:[1,\infty)\to [1,\infty)$ be nondecreasing. We have 
$$\int_1^\infty (t\psi(t)^2)^{-1} dt=\infty$$
 if and only if 
 $$\sum_{j=1}^\infty \psi(b^j)^{-2}=\infty $$
 for any $b>1$. 
\end{coro}

\subsection{Convergence reduction}  \label{sec:convcase}
\begin{lem}\label{lem:convergence reduction}
{Suppose $\psi:[1,\infty)\to[1,\infty)$ is nondecreasing and also that $\int_1^\infty \frac{1}{t\psi(t)^2} \, dt <\infty$.} If %If whenever $\int_1^\infty \frac{1}{t\psi(t)^2} \, dt <\infty$, 
$$\mu\left(\{\omega: \liminf_{R\to\infty} \psi(R) R^2 \zeta_\omega(R) >0\} \right) =1,$$ %then whenever  $\int_1^\infty \frac{1}{t\psi(t)^2} \, dt <\infty$ 
then $$\mu\left(\left\{\omega: \liminf_{R\to\infty} \psi(R) R^2 \zeta_\omega(R) =\infty\right\} \right) = 1.$$
\end{lem}
\begin{proof}
	We first construct a slightly smaller nondecreasing function $\psi_0(t)$ so that 
	$$\int_1^\infty \frac{1}{t{\psi_0}(t)^2} \,dt  <\infty\text{ and } \lim_{t\to\infty} \frac{\psi(t)}{{\psi_0}(t)} = \infty.$$
	To do this let $n_1 = 1$ and set 
	$$n_j = \sup\left\{T > n_{j-1}:\int_{n_{j-1}}^{T} \frac{1}{t\psi(t)^{2}} \,dt \leq 2^{-j}\int_1^\infty \frac{1}{t\psi(t)^{2}}\,dt.\right\}$$ For $j\in \mathbb{N}$ we piecewise define
	$${\psi_0}(t) = j^{-1} \psi(t) \text{ whenever } n_j \leq t < n_{j+1}.$$
	Observe that $\psi_0$ is also nondecreasing. 
	Then since $j\to \infty$ as $t\to\infty$, we have $\lim_{t\to\infty} \psi(t) / \psi_0(t) = \infty$. By our choice of $n_j$, we have $\int_1^\infty (t \psi_0(t)^2)^{-1} \,dt < \infty.$ By the assumption of Lemma~\ref{lem:convergence reduction} there is a full measure set of $\omega$ so that $\liminf_{t\to\infty} \psi_0(t) t^2 \zeta_\omega(t)>0$. From this we have the desired result that for a full measure set of $\omega$,
	$$\liminf_{t\to\infty} \psi(t) t^2 \zeta_\omega(t) \geq \left(\liminf_{t\to\infty} \frac{\psi(t)}{\psi_0(t)}\right) \left(\liminf_{t\to \infty} \psi_0(t) t^2 \zeta_\omega(t)\right) = \infty.$$\qed
	\end{proof}

\subsection{Divergence Case} \label{sec:divcase} Recall for a sequence of sets $(A_k)_{k=1}^\infty$, the limit superior is given by $\limsup A_i= \cap_{N=1}^\infty \cup_{i=N}^\infty A_i $.

\begin{prop}\label{prop:div improvement}
If for all $\sigma >0$, $\mu(\limsup A_k(0,\sigma,\psi,b))>0$ then $$\liminf_{R\to\infty}  \psi(R) R^2\zeta_\omega(R)=0$$ for $\mu$-a.e. $\omega$.
\end{prop}
\begin{rmk} \label{rmk:propdiv}
	Note that $A_k(0, \sigma) \subseteq A_k(0, \sigma')$ whenever $\sigma < \sigma'$. Thus the assumption of Proposition~\ref{prop:div improvement} is satisfied as long as $\mu(\limsup A_k(0,\sigma)) >0$ for all $\sigma$ small enough. 
\end{rmk} 
\begin{proof}
Fix $\sigma$ and write $A_k$ for $A_k(0,\sigma)$. We first claim $\mu(\limsup A_k) = 1$. We will show $\limsup A_k$ is invariant under forward geodesic flow $g_{\log(b^t)}$ for $t>0$. Let $$\omega \in g_{\log(b^{-t})} \limsup_{k\to\infty} g_{\log(b^k)} H_{0,\sigma, k}.$$ For $n\in \mathbb{N}$, there exists $k\geq n+t$ so that the monotonicity of $\psi$ implies
$$\omega \in g_{\log(b^{-t+k})} H_{0,\sigma,k} \subseteq g_{\log(b^{-t+k})} H_{0,\sigma, -t+k}.$$ Hence $\omega \in \limsup A_k$. By ergodicity of the geodesic flow and the assumption that $\mu(\limsup A_k) >0$, we conclude that $\mu(\limsup A_k) = 1$.

We now translate $\limsup A_k$ % by a small backwards geodesic flow to a set where we have the desired result. Namely 
and claim that for $s_0=\left\lceil \frac{\log(2)}{2\log(b)} \right\rceil >0$, any $\widetilde{\omega} \in g_{-s_0\log(b)} \limsup A_k$ satisfies
$$\liminf_{R\to\infty}\psi(R) R^2 \zeta_{\widetilde{\omega}}(R) = 0.$$
First note that we are still working with a full measure set since $\mu$ is invariant under geodesic flow
$$\mu(g_{-s_0 \log(b)} \limsup A_k) = \mu(\limsup A_k) = 1.$$
Let $\tilde{\omega} = g_{-s_0 \log(b)} \omega$ for some $\omega \in \limsup A_k$. Since $\omega \in \limsup A_k$, for any $m\in \N$ we can find $\rho_m \geq m$ so that $\omega \in g_{\rho_m \log(b)} H_{0,\sigma, \rho_m}.$ The choice of $s_0$ guarantees that the longest possible holonomy vectors in $g_{-s_0 \log(b)} T^\pm_{0, \sigma, b^{\rho_m}}$ are at most $b^{\rho_m}$ since the choice of $s_0$ is sufficient so that
$$b^{2(\rho_m-s_0)} + \frac{\sigma^2}{b^{2(\rho_m-s_0)} \psi(b^{\rho_m})^2} \leq b^{2\rho_m}.$$ 
Thus the holonomy vectors detected by $g_{-s_0 \log(b)} T^\pm_{0, \sigma, b^{\rho_m}}$ have length at most $b^{\rho_m}$ and an upper bound on the angle around zero, giving the following upper bound 
$$\zeta_{\tilde{\omega}} (b^{\rho_m}) \leq \arctan \frac{\sigma}{\psi(b^{\rho_m}) b^{2\rho_m - 2s_0}} < \frac{\sigma b^{2s_0}}{\psi(b^{\rho_m}) b^{2\rho_m}}.$$
Since $s_0$ is fixed, we can take $\sigma \to 0$ to obtain $\liminf_{R\to\infty}  \psi(R) R^2\zeta_\omega(R) = 0$.\qed
\end{proof} 

\subsection{Axiomatic framework} \label{sec:axiom}

We will first recall the Borel--Cantelli lemma, and then spend the remainder of the section stating and proving a partial converse.

\begin{lem}[Borel--Cantelli lemma]\label{lem:BC}
	Suppose $(A_k)_{k=1}^\infty$ are measurable sets with $\sum_{k=1}^\infty \mu(A_k) < \infty$. Then $\mu( \limsup A_i)=0.$
\end{lem}
\begin{prop}[Exponential decay converse to Borel--Cantelli] \label{prop:Exp_Decay_Borel_Cantelli}
Let $C\geq 1$, $0<\delta < 1$, and $(A_k)_{k=1}^\infty$, {$(B_k)_{k=1}^\infty$, $(C_k)_{k=1}^\infty$} be measurable sets. Suppose the following hold.
\begin{enumerate}
\item \label{enum:1}$\sum_{k=1}^\infty \mu(A_k)=\infty$.
\item \label{enum:2}For all $i \leq j$, $\mu(A_i)\geq \mu(A_j).$
\item \label{enum:3}For all $i$, for all $j$ so that $j>i+ C \log\left(\frac 1 {\mu(A_i)}\right)$ we have 
$$\mu(A_i\cap A_j)\leq C\mu(A_i)\left[\mu(A_j)+e^{-\frac{\delta}{4} |i-j|}\right].$$
\item \label{enum:4} For all $i$, for all $j$ so that $i< j\leq i+C\log\left(\frac 1 {\mu(A_i)}\right)$
\begin{enumerate}
\item \label{enum:4(a)}$B_i\subset A_i$ and $A_j \subset C_j$,
\item \label{enum:4(b)}$\mu(B_i)>\frac 1 C \mu(A_i)$,
\item \label{enum:4(c)}$\mu(C_j)< C \mu(A_j)^{\frac 1 2}$,
\item \label{enum:4(d)}$\mu(B_i \cap C_j)<C \mu(B_i)\left(2^{-(j-i)(1-\delta)} +\mu(C_j)^\frac{1+\delta}{2}\right)$.
\end{enumerate}
\end{enumerate}
Then $\limsup A_i$ has positive measure.
\end{prop}

\begin{rmk}\label{rmk:technical} One can observe the decay of correlations in Assumption~\eqref{enum:3} is insufficient to handle the generality of Theorem \ref{thm:01law}. Indeed consider $\psi(t)=\sqrt{\log(t+4)\log(\log(t+4))}$. Since $\int_1^\infty \frac 1 {t\psi(t)^2} \,dt=\infty$, $\psi$ satisfies the second assumption of Theorem~\ref{thm:01law}. Motivated by Lemma~\ref{lem:measurebounds} we assume $\mu(A_k)$ is proportional to $\psi(b^{k})^{-2} = \frac{1}{\log(b^k +4) \log(\log(b^k+4))} < \frac{1}{k \log(k) \log(b)}$. {If $C'>0$ is small enough, }assumption~\eqref{enum:3} gives no bounds on $\mu(A_i\cap A_j)$ when $i<j<C'\log(i)+i$. Define $n_k$ recursively by $n_1=2$ and $n_{i+1}=\lceil n_i+C'\log(n_i)\rceil$. Observe
$\sum_{i=1}^\infty \mu(A_{n_i})<\infty$, so $\mu(\limsup A_{n_k}) =0$ by the Borel--Cantelli lemma. Thus we cannot draw any conclusions on $\mu(\limsup A_k)$ without Assumption~\eqref{enum:4}. 
\end{rmk}

We want to prove Proposition~\ref{prop:Exp_Decay_Borel_Cantelli}. This proof is inspired by \cite[Theorem~2.1]{Petrov02}, which invokes the Chung--Erd\H{o}s inequality.

\begin{lem}[Chung--Erd\H{o}s inequality \cite{ChungErdos}] \label{lem:CEI}
	Suppose$(A_k)_{k=1}^\infty$ is a sequence of measurable sets with $\mu\left(\bigcup_{k=1}^N A_k\right) > 0$, then
	\begin{equation}\label{eq:CEinequality}\mu\left(\bigcup_{k=1}^N A_k\right) \geq \frac{\left(\sum_{k=1}^N \mu(A_k)\right)^2}{\sum_{j,k = 1}^n \mu(A_j\cap A_k)}.
	\end{equation}
\end{lem}
%\begin{theopargself}
%\begin{proof}[of Lemma~\ref{lem:CEI}.]%
%	Beginning with the numerator on the right hand side,
%	$$\left(\sum_{k=1}^N \mu(A_k)\right)^2 = \left[\int\chi_{\{\sum_{k=1}^N {\chi_{A_k} > 0}\}}  \left(\sum_{k=1}^N \chi_{A_k} \right)\,d\mu  \right]^2.$$
%	Notice
%	$$\int \chi_{\{\sum {\chi_{A_j} > 0}\}}^2 \,d\mu =\int \chi_{\{\sum {\chi_{A_j} > 0}\}}\,d\mu = \mu\left(\bigcup_{j=1}^N A_j\right),$$
%	By the Cauchy--Schwarz inequality
%	\begin{align*}\left(\sum_{k=1}^N \mu(A_k)\right)^2  &\leq \mu\left(\bigcup_{k=1}^N A_k\right)\int \left(\sum_{k=1}^N \chi_{A_k}\right)^2 \,d\mu\\  &= \mu\left(\bigcup_{k=1}^N A_k\right)\left[\sum_{j,k = 1}^N \mu(A_j\cap A_k) \right]. 
%	\end{align*}
%	Rearranging we obtain the desired inequality.\qed
%\end{proof}
%\end{theopargself}

The next lemma will also be used to prove Proposition~\ref{prop:Exp_Decay_Borel_Cantelli}, which states the conditions on the sets $B_k \subseteq A_k$ that we use to show $\limsup(B_k) > 0$, and thus $\limsup(A_k) > 0$.
\begin{lem} \label{lem:EBC_Bs}
	{We retain the notation of Proposition~\ref{prop:Exp_Decay_Borel_Cantelli}}.
There exists some $\tilde{C}\geq 1$ depending only on the constant $C$ {as in Proposition~\ref{prop:Exp_Decay_Borel_Cantelli}} large enough so that if $\tilde{m}_i = i + \tilde{C} + \tilde{C}\log\left(\frac{1}{\mu(B_i)}\right)$, the following hold.
	\begin{enumerate}[label=(\alph*)]
		\item \label{enum:1'} $\sum_{k=1}^\infty \mu(B_k) = \infty.$ 
		\item \label{enum:2'}For all $i \leq j$, $\mu(B_i) \geq \frac{1}{C}\mu(B_j).$
		\item \label{enum:3'}   For all $i$ and for all $j$ so that $j > \tilde{m}_i$,
		$$\mu(B_i\cap B_j) \leq \tilde{C}\mu(B_i)\left[\mu(B_j) + e^{-\frac{\delta}{4}|i-j|}\right].$$
		\item \label{enum:4'} For all $i$ and for all $j$ with $i< j < \tilde{m}_i$,
		$$\mu(B_i\cap B_j) < \tilde{C} \mu(B_i)\left[2^{-|i-j|(1-\delta)} + \mu(B_j)^\frac{1+\delta}{4}\right].$$
	\end{enumerate}
	Moreover there exists constants $D, D'>0$ so that for any $n>0$ and $N\geq n$,
	\begin{equation}\label{eq:upperbounden}	
		\sum_{i,j=n}^N \mu(B_i\cap B_j) \leq \tilde{C}\left[D\sum_{k=n}^N \mu(B_k) + D' + \left(\sum_{k=n}^N \mu(B_k)\right)^2\right].
	\end{equation}
\end{lem}
\begin{theopargself}\begin{proof}[of Lemma~\ref{lem:EBC_Bs}.] The first 4 parts use assumptions on $A_k$ in Proposition~\ref{prop:Exp_Decay_Borel_Cantelli}.
	\begin{enumerate}[label=(\alph*)]
		\item Follows from Assumptions~\eqref{enum:1} and \eqref{enum:4(b)}.
		\item Follows from Assumptions~\eqref{enum:2} and \eqref{enum:4(b)}.
		\item Combine Assumption~\eqref{enum:3} with Assumptions~\eqref{enum:4(a)} and \eqref{enum:4(b)}.
		\item Combine Assumptions~\eqref{enum:4(b)} and \eqref{enum:4(c)} with that fact that $\frac{1+\delta}{4} \in (\frac{1}{4}, \frac{1}{2})$ and $C\geq 1$ to obtain
$$\mu(C_j)^\frac{1+\delta}{2} \leq C \mu(A_j)^\frac{1+\delta}{4} \leq C(C\mu(B_j))^\frac{1+\delta}{4} \leq \tilde{C}^2\mu(B_j)^\frac{1+\delta}{4}.$$
\end{enumerate}

Now we move to Equation~\eqref{eq:upperbounden} where we want to find an upper bound for the denominator on the right hand side of Equation~\eqref{eq:CEinequality} applied to the sets $B_k$. First since $\tilde{C} \geq 1$,
\begin{equation} \label{eq:EBC1}
	\sum_{i,j=n}^N \mu(B_i\cap B_j) \leq  2\sum_{i=n}^N \sum_{j>i} \mu(B_i\cap B_j) + \tilde{C} \sum_{k=n}^N \mu(B_k).\\
%	&\leq \tilde{C}\left[\sum_{k=n}^N \mu(B_k) + 2\sum_{i=n}^N \sum_{j>\tilde{m}_i}^N\mu(B_i)\mu(B_j) + \mu(B_i)e^{-\frac{\delta}{4}|i-j|}\right.  \\
%	&\hspace{1.4cm}+ \left.2\sum_{i=n}^N \sum_{i<j\leq\tilde{m}_i} \mu(B_i) 2^{-|i-j|(1-\delta)} + \mu(B_i)\mu(B_j)^\frac{1+\delta}{4}  \right] \tag{By \ref{enum:3'} and \ref{enum:4'} and since $\tilde{C} \geq 1$}\\
%	&\leq \tilde{C}\left[\sum_{k=n}^N \mu(B_k) + \left(\sum_{k=n}^N \mu(B_k)\right)^2 + \frac{2}{1-e^{-\frac{\delta}{4}}} \sum_{k=n}^N \mu(B_k) \right.\\
%	&\hspace{1.4cm}+ \left. \frac{2}{1-2^{-(1-\delta)}} \sum_{k=n}^N \mu(B_k) + 2C\tilde{C} \sum_{k=n}^N \mu(B_k) + 2C\tilde{C}C' + 2C\tilde{C} \sum_{k=n}^N \mu(B_k)  \right]\tag{By justifications 1,2,3,4 below.}\\
%	&= \tilde{C}\left[D\sum_{k=n}^N \mu(B_k) + D' + \left(\sum_{k=n}^N \mu(B_k)\right)^2\right] 
\end{equation}

We split the double sum on the right hand side of Equation~\eqref{eq:EBC1} into cases when $j > \tilde{m}_i$ and $j\leq \tilde{m}_i$. 

In the first case when $j > \tilde{m}_i$, part \ref{enum:3'} implies
\begin{equation}\label{eq:EBCfar1}
	2\sum_{i=n}^N \sum_{j>\tilde{m}_i} \mu(B_i\cap B_j)  \leq 2\tilde{C} \sum_{i=n}^N \sum_{j>\tilde{m}_i}\left(\mu(B_i)\mu(B_j) + \mu(B_i) e^{-\frac{\delta}{4}|i-j|}\right).
\end{equation}
In the first sum of Equation~\eqref{eq:EBCfar1}, we re-expand the square so that
\begin{equation}\label{eq:EBCfar2}
2\sum_{i=n}^N \sum_{j> \tilde{m}_i}^N \mu(B_i)\mu(B_j) \leq  \left(\sum_{i=n}^N \mu(B_i)\right)^2 - \sum_{i=n}^N \mu(B_i)^2 \leq \left(\sum_{i=n}^N \mu(B_i)\right)^2.
\end{equation}
In the second sum of Equation~\eqref{eq:EBCfar1} we making the change of variables $k= j-i$, the geometric series formula gives
\begin{equation}\label{eq:EBCfar3}
2\sum_{i=n}^N \sum_{j>\tilde{m}_i}^N \mu(B_i) e^{-\frac{\delta}{4}|i-j|} \leq 2\sum_{i=n}^N \mu(B_i) \frac{1}{1-e^{-\frac{\delta}{4}}}.
\end{equation}
Combining Equation~\eqref{eq:EBCfar1}, Equation~\eqref{eq:EBCfar2}, and Equation~\eqref{eq:EBCfar3}, we obtain
\begin{equation} \label{eq:EBC2}
	2\sum_{i=n}^N \sum_{j>\tilde{m}_i} \mu(B_i\cap B_j) \leq \tilde{C}\left[\left(\sum_{k=n}^N \mu(B_k)\right)^2 + \frac{2}{1- e^{-\frac{\delta}{4}}} \sum_{k=1}^N \mu(B_k)\right].
\end{equation}

We now consider the case for $i< j \leq \tilde{m}_i$. By \ref{enum:4'}
\begin{align}\nonumber
	&2\sum_{i=n}^N \sum_{i<j\leq \tilde{m}_i}\mu(B_i\cap B_j) \\
	&\leq 2\tilde{C} \sum_{i=n}^N \sum_{i<j\leq \tilde{m}_i}\left(\mu(B_i)2^{-|i-j|(1-\delta)} + \mu(B_i) \mu(B_j)^{\frac{1+\delta}{4}}\right).\label{eq:EBCnear1}
\end{align}
We again use a geometric series formula so that the first sum is bounded by
\begin{equation}\label{eq:EBCnear2}
	2\sum_{i=n}^N  \sum_{i<j\leq \tilde{m}_i}\mu(B_i)2^{-|i-j|(1-\delta)}  \leq \frac{2}{1-2^{-(1-\delta)}} \sum_{i=n}^N \mu(B_i).
\end{equation}

Note $\frac{5+\delta}{4} \in \left(\frac{5}{4}, \frac{6}{4}\right)$ so $\frac{5+\delta}{4}$ is a power bigger than 1 with $\mu(B_i)\leq 1$. Therefore $\mu(B_i)^\frac{5+\delta}{4} \leq \mu(B_i).$ Combining this fact with \ref{enum:2'},
\begin{align}
&2\sum_{i=n}^N \sum_{i<j\leq \tilde{m}_i} \mu(B_i)\mu(B_j)^\frac{1+\delta}{4} \nonumber\\
&\leq 2C\sum_{i=n}^N \mu(B_i)^\frac{5+\delta}{4} (\tilde{m}_i-i) \nonumber \\ 
&\leq 2C\tilde{C} \sum_{i=n}^N \mu(B_i) + 2C\tilde{C} \sum_{i=n}^N \mu(B_i)^\frac{5+\delta}{4} \log\left(\frac{1}{\mu(B_i)}\right). \label{eq:EBCnear3}
\end{align}
Now choose $n_0$ large enough so that for all $i\geq n_0$, $$\log\left(\frac{1}{\mu(B_i)}\right) \leq \mu(B_i)^\frac{-1-\delta}{4}.$$
Then if $n\geq n_0$, 
$$\sum_{i=n}^N \mu(B_i)^\frac{5+\delta}{4}\log\left(\frac{1}{\mu(B_i)}\right) \leq \sum_{i=n}^N \mu(B_i).$$
Otherwise if $n\leq n_0$
\begin{align*}
\sum_{i=n}^N \mu(B_i)^\frac{5+\delta}{4}\log\left(\frac{1}{\mu(B_i)}\right) &\leq \sum_{i=n}^{n_0-1} \mu(B_i) \log\left(\frac{1}{\mu(B_i)}\right) + \sum_{i=n_0}^N \mu(B_i)\\
& \leq C' + \sum_{i=n}^N \mu(B_i).
\end{align*}
where $C' >0$ is the bound for the finite sum.
Therefore there is a constant $C'$ so that
\begin{equation} \label{eq:EBCnear4}
\sum_{i=n}^N \mu(B_i)^\frac{5+\delta}{4} \log\left(\frac{1}{\mu(B_i)}\right)\leq C' + \sum_{i=n}^N \mu(B_i).
\end{equation}
Thus we conclude by combining Equation~\eqref{eq:EBCnear1}, Equation~\eqref{eq:EBCnear2}, Equation~\eqref{eq:EBCnear3}, and Equation~\eqref{eq:EBCnear4} so that
\begin{equation}\label{eq:EBC3}
	2\sum_{i=n}^N \sum_{i<j\leq \tilde{m}_i} \mu(B_i \cap B_j) \leq \tilde{C}\left[2CC'\tilde{C}+\left(\frac{2}{1-2^{-(1-\delta)}} + 4C \tilde{C}\right) \sum_{i=n}^N \mu(B_i) \right].
\end{equation}
The proof of Equation~\eqref{eq:upperbounden} is completed by combining Equation~\eqref{eq:EBC1}, Equation~\eqref{eq:EBC2}, and Equation~\eqref{eq:EBC3} where $D' = 2C\tilde{C}C'$ and $$D= 1+\frac{2}{1-e^{-\frac{\delta}{4}}} + \frac{2}{1-2^{-(1-\delta)}} + 4C\tilde{C}.$$\qed
\end{proof}\end{theopargself}
\begin{theopargself}
\begin{proof}[of Proposition~\ref{prop:Exp_Decay_Borel_Cantelli}.]
	It suffices show that the measure of $\limsup B_k$ has positive measure. By Chung--Erd\H{o}s inequality (Lemma~\ref{lem:CEI}), Lemma~\ref{lem:EBC_Bs}~\ref{enum:1'}, and Equation~\eqref{eq:upperbounden},
\begin{align*}
	\liminf_{N\to\infty}\mu\left(\bigcup_{k=n}^N B_k\right) &\geq \liminf_{N\to\infty}\frac{\left(\sum_{k=n}^N \mu(B_k)\right)^2}{\tilde{C} \left[D\sum_{k=n}^N\mu(B_k) + D' + \left(\sum_{k=n}^N \mu(B_k)\right)^2\right]}\\
	&= \liminf_{N\to\infty}\frac{1}{\tilde{C}\left[ \frac{D}{\sum_{k=n}^N \mu(B_k)} + \frac{D'}{\left(\sum_{k=n}^N \mu(B_k)\right)^2} + 1\right]}= \frac{1}{\tilde{C}}.
\end{align*}
Hence $\mu\left(\bigcup_{k=n}^\infty B_k\right) \geq \frac{1}{\tilde{C}}.$ Notice $\bigcup_{k=n}^\infty B_k$ is a nested decreasing sequence of sets, so we conclude
$$\mu(\limsup B_n) = \mu\left(\bigcap_{n=1}^\infty\bigcup_{k=n}^\infty B_k\right) = \lim_{n\to\infty} \mu\left(\bigcup_{k=n}^\infty B_k\right)\geq \frac{1}{\tilde{C}} > 0.$$\qed \end{proof}\end{theopargself}

\subsection{Proof of Theorem \ref{thm:01law}} \label{sec:proof01law}
We now prove Theorem \ref{thm:01law} conditional on the next proposition, whose proof is completed in Section~\ref{sec:ExpDecayFar} and Section~\ref{sec:averages}.

\begin{prop}\label{prop:sufficient is satisfied}
For some $c>0$ and for arbitrarily small $\sigma >0$, there exist sets $B_k$, $C_k$ so that along with the sets $A_k = A_k(c,\sigma)$ as defined in Definition~\ref{def:As}, the assumptions of Proposition~\ref{prop:Exp_Decay_Borel_Cantelli} are satisfied.
\end{prop}
Before outlining how to prove Proposition~\ref{prop:sufficient is satisfied}, we first show that it is sufficient to obtain Theorem~\ref{thm:01law}.

\begin{theopargself}\begin{proof}[of Theorem~\ref{thm:01law}.]
We first prove the convergence case. By Corollary~\ref{cor:psidivergence}, $\sum_{j=1}^\infty \psi(e^j)^{-2} < \infty$, and thus $\lim_{j\to\infty} \psi(e^j) = \infty$. Set
$$L_k = \{\omega: \exists \,R \in [e^k,e^{k+1}] \text{ so that }R^2\zeta_\omega(R)< \psi(e^k)^{-1}\}.$$ For $k$ large enough, if $\zeta_\omega(R) <  \frac{\psi(e^k)^{-1}}{R^2}$, then there are two saddle connections on $\omega$ with horizontal holonomy at most $R$ and vertical holonomy of magnitude at most $\frac{\psi(e^k)^{-1}}{R}$. It follows that for $\omega \in L_k$, $g_{-\log\left(e^k \sqrt{ \psi(e^k)}\right) } \omega$ has two saddle connections of length at most $\psi(e^k)^{-1/2} e$. By Lemma~\ref{lem:MasurSmillie}, there is some constant $C'$ so that $\mu(L_k) \leq C'  \psi(e^k)^{-2}$ so that $\sum_{k=1}^\infty \mu(L_k) < \infty$. By the Borel--Cantelli lemma (Lemma~\ref{lem:BC}), $\mu(\limsup L_k) = 0$. Taking the complement, we have a full measure set of $\omega$ so that for all $k \geq k_0$ (where $k_0$ depends on $\omega$), and for all $R \in [e^k, e^{k+1}]$, 
$$\psi(R)R^2 \zeta_\omega(R) \geq \psi(e^k) R^2 \zeta_\omega(R) \geq 1.$$ Therefore
\begin{align*} \mu\left(\{\omega: \liminf_{R\to\infty} \psi(R) R^2 \zeta_\omega(R) > 0\}\right)
\geq \mu\left(\{\omega: \liminf_{R\to\infty} \psi(R) R^2 \zeta_\omega(R) \geq 1\}\right) = 1.\end{align*}
By Lemma~\ref{lem:convergence reduction}, the convergence case of Theorem~\ref{thm:01law} is verified.

We now prove the divergence case. By Corollary~\ref{cor:psidivergence}, $\sum_{k=1}^\infty\psi(b^k)^{-2}=\infty$. By Lemma~\ref{lem:measurebounds} for any $\sigma <\sigma_{\mathcal{H}}$ our set $A_k$ has measure proportional to $\psi(b^k)^{-2}$. By Proposition~\ref{prop:Exp_Decay_Borel_Cantelli} and Proposition~\ref{prop:sufficient is satisfied}, for a positive measure set of $\omega$ we have $g_t\omega \in A_k$ for infinitely many $k$. Following Remark~\ref{rmk:propdiv} we satisfy the assumptions of Proposition \ref{prop:div improvement}, which implies the divergence case of Theorem~\ref{thm:01law}. \qed
\end{proof}\end{theopargself}

\begin{theopargself} \begin{proof}[outline of Proposition~\ref{prop:sufficient is satisfied}.]

	We verify or state where each of the assumptions of Proposition~\ref{prop:Exp_Decay_Borel_Cantelli} are verified.
\begin{enumerate}
	\item Assumption~\eqref{enum:1} follows by the measure bounds of Lemma~\ref{lem:measurebounds}, Corollary~\ref{cor:psidivergence} and the fact that $g_t$ preserves measure: 
			$$\sum_{k=1}^\infty \mu(A_k) \geq \sum_{k=1}^\infty m \frac{\sigma^2}{\psi(b^k)^2} = \infty.$$
	\item Assumption~\eqref{enum:2} follows by Lemma~\ref{lem:measurebounds}: $i \leq j$ implies $\psi(b^i) \leq \psi(b^j)$, so 
			$$\mu(A_i) = \frac{m\sigma^2}{\psi(b^i)^2} \geq \frac{m\sigma^2}{\psi(b^j)^2} = \mu(A_j).$$
			\item Assumption~\eqref{enum:3} is proved in Section~\ref{sec:(3)} using Corollary~\ref{cor:AGY_for_rho}, Lemma~\ref{lem:intersection less than rho product}, and Lemma~\ref{lem:productrhos_sumE_i}.
			\item Assumption~\eqref{enum:4} is proved as follows. The construction of the $B_k$ and $C_k$ sets along with proofs of Assumptions~\eqref{enum:4(a)}, \eqref{enum:4(b)}, and \eqref{enum:4(c)} are given in Section~\ref{subsec:4a-c}. The proof is completed by verifying Assumption~\eqref{enum:4(d)} in Section~\ref{subsec:4d}.
\end{enumerate}\qed
\end{proof}\end{theopargself}
	
\section{Verifying Proposition~\ref{prop:Exp_Decay_Borel_Cantelli} Assumption (3): exponential decay of correlations for far away pairs}\label{sec:ExpDecayFar}

We begin by stating our key exponential mixing result:

\begin{thm}[Stated from \cite{athreya2012right} Theorem C.4, see \cite{avila2006exponential}] \label{thm:AGY}
	Fix $\mathcal{H}$ a connected component of the stratum and let $\mu$ be MSV measure as above. There exist $C>0$ and $\delta > 0$ so that for all $h_1$, $h_2$ Lipschitz and compactly supported in a compact set $K$, there exists a constant $C_K$ depending only on the compact set $K$ so that for all $t\geq 0$
	\begin{align*}&\abs{\int h_1 (h_2\circ g_t) \,d\mu - \int h_1 \,d\mu \int h_2\,d\mu }\\ 
	&\leq C(C_K + \norm{h_1}_\infty + \norm{h_2}_{Lip}) (C_K + \norm{h_2}_\infty + \norm{h_2}_{Lip}) e^{-\delta t}.\end{align*}
\end{thm}

In order to apply Theorem~\ref{thm:AGY}, we need to use bump functions to approximate the sets $A_i$ and $A_j$. By $g_t$-invariance of $\mu$, $$\mu(A_i\cap A_j) = \mu(H_{c_{\mathcal{H}},\sigma, i}\cap g_{\log(b^{j-i})} H_{c_{\mathcal{H}},\sigma,j}),$$ {and in particular, using Theorem \ref{thm:AGY} to bound $\mu(H_{c_{\mathcal{H}},\sigma, i}\cap g_{\log(b^{j-i})} H_{c_{\mathcal{H}},\sigma,j})$ in turn gives bounds for $\mu(A_i \cap A_j)$.} We will thus define our bump function to approximate $H_{c_{\mathcal{H}},\sigma, i}$ and $H_{c_{\mathcal{H}},\sigma, j}$.

\begin{defn}
Recall we fixed above (Section~\ref{sec:defmeas}) $0<\delta<1$ as in Theorem~\ref{thm:AGY}. For each $i$ and each $j > i + \frac{4}{\delta}\log\left(\frac{1}{\mu(A_i)}\right)$, define
$$\epsilon_{i,j} = e^{-\frac{\delta}{4}|i-j|}.$$
 Then for $\ell \in \{i,j\}$, define $\rho_{i,j}^{\ell}: \mathcal{H}\to \R$ to be zero outside of the coordinate patch for $H_{c_{\mathcal{H}},\sigma, i}$ and $H_{c_{\mathcal{H}},\sigma, j}$. 
Notice that by our choice, all the $H_{c_{\mathcal{H}},\sigma, k}$ are contained in a single coordinate patch, 
 $x_1,x_2\in \C$, and $x_3 \in \C^{2g+s-3}$ we equip each copy of $\C^\ell$ with the standard $L^2$ Euclidean metric on $\R^{2\ell}$. On this coordinate patch we set 
$$\rho_{i,j}^{\ell}(x_1,x_2,x_3) = f_1^{\ell}(x_1)f_2^{\ell}(x_2)f_3^{\ell}(x_3)$$
where
$$f_1^{\ell}(x_1) = \min\left\{1, \frac{1}{\epsilon_{i,j}} \dist\left(x_1,\bd T^+_{c_{\mathcal{H}},\sigma, \psi(b^\ell)}\right)\right\}\cdot \chi_{T^+_{c_{\mathcal{H}},\sigma, \psi(b^\ell)}},$$
$$f_2^{\ell}(x_2) = \min\left\{1, \frac{1}{\epsilon_{i,j}} \dist\left(x_2,\bd T^-_{c_{\mathcal{H}},\sigma, \psi(b^{\ell})}\right)\right\} \chi_{T^-_{c_{\mathcal{H}},\sigma, \psi(b^{\ell})}},$$
and
$$f_3^{\ell}(x_3) = \min\left\{1, \frac{1}{\epsilon_{i,j}} \dist \left(x_3, \bd B(0,1)\right) \right\} \chi_{B(0,1)}.$$
\end{defn}

 {We want to show that our approximation functions are Lipschitz, so we first discuss some options of norms.} The $\sup$-norm metric, which we denote $d_{AGY}$, is used to define the Lipschitz norm in Theorem~\ref{thm:AGY} which is given in \cite[Section 2.2.2]{avila2006exponential}. For our purposes it is more convenient to use a metric coming from period coordinates, $d_{per}$. We refer the reader to \cite[Section 3.2]{Per-Teich} {(where $d_{per}$ is denoted $d_{Euclidean}$)} for a detailed description of the construction on {the space of all quadratic differentials with a fixed genus}. {We note that the arguments in  \cite[Section 3.2]{Per-Teich} can be repeated verbatim for a fixed stratum.} On each compact set this metric is uniformly comparable to the $\sup$-norm metric, so for the estimate above the only effect on the Lipschitz functions we consider, which are all supported in a fixed compact set, is a multiplicative constant in the Lipschitz norms. 

We now briefly outline that these two metrics are uniformly comparable on compact sets.
 Both $d_{AGY}$ and $d_{per}$ are defined as path metrics coming from norms on relative cohomology, where the norms depend on the point. Let us call these \emph{AGY norms} and \emph{Period norms}. The AGY norms vary continuously \cite[Proposition 2.11]{avila2006exponential} and so for each compact $\hat{K}$  all the AGY norms corresponding to points in $\hat{K}$ are uniformly comparable. On each compact set there are only finitely many Period norms (see the Remark at the end of \cite[{Section 3.2}]{Per-Teich}), so all the norms we consider are uniformly comparable. The diameter of each compact set, $K$, is bounded in both $d_{per}$ and $d_{AGY}$, so it suffices to prove there is a $\hat{\delta}>0$ and $C$ (both depending on ${K}$) so that if $x,y \in {K}$ then 
\begin{equation}\label{eq:met comp}
\frac 1 C \min\{\hat{\delta}, d_{per}(x,y)\}< \min\{\hat{\delta}, d_{AGY}(x,y)\}<C \min\{\hat{\delta}, d_{per}(x,y)\}.
\end{equation} 
To see \eqref{eq:met comp}, we choose a slightly larger compact set ${K}'$ so that ${K}\subset Int({K}')$ and in particular $\min\{d_{per}({K},\partial{K}'),d_{AGY}({K},\partial{K}')\}=\delta'>0$. Let $C'$ be so that all the AGY and Period norms for all $\omega \in {K}'$ are comparable by $C'$. That is, if $\omega, \omega' \in {K}'$ and $\|\cdot\|_{\omega},\|\cdot\|_{\omega'}$ are Period/AGY norms at $\omega,\omega'$ then 
$$\frac 1 {C'}\|x-y\|_{\omega}\leq \|x-y\|_{\omega'}\leq C'\|x-y\|_{\omega}.$$ We choose $\hat{\delta}<\frac 1 {C'}\delta'$ so that for any $\omega \in {K}$, the $\hat{\delta}$-neighborhood of $\omega$ in both metrics is contained in a coordinate patch. Let $\omega \in {K}$ be given and $dev:\mathcal{U} \to H^1(S,\Sigma,\mathbb{C})$ be a coordinate chart so that the $\hat{\delta}$-neighborhood of $\omega$ in both metrics are contained in $\mathcal{U}$. Now if $\min\{d_{per}(\omega,\omega'),d_{AGY}(\omega,\omega')\}<\hat{\delta}$ we have that 
$$\frac 1 {C'}\|dev(\omega)-dev(\omega')\|_{\omega''}<d_{per}(\omega,\omega')<C'\|dev(\omega)-dev(\omega')\|_{\omega''}$$
and 
$$\frac 1 {C'}\|dev(\omega)-dev(\omega')\|_{\omega''}<d_{AGY}(\omega,\omega')<C'\|dev(\omega)-dev(\omega')\|_{\omega''}$$
where $\omega''\in {K}'$ is arbitrary and $\|\cdot\|_{\omega''}$ is either of the AGY or Period norms at $\omega''$. Because \eqref{eq:met comp} is true outside of a $\hat{\delta}$ neighborhood in both metrics, we have \eqref{eq:met comp} with $C=C'$.

\begin{lem}\label{lem:rhoLipschitz} There exists a constant $C'$ so that the functions $\rho_{i,j}^\ell$ are $\frac{{C'}}{\epsilon_{i,j}}$-Lipschitz with respect to $d_{per}$.
\end{lem} 
\begin{proof}Note that $f_{k}^\ell$ for $k=1,2,3$ are all $\frac{1}{\epsilon_{i,j}}$-Lipschitz as the Euclidean distance function being $1$-Lipschitz implies 
\begin{align*}\left|f_k^\ell(x_k) - f_k^\ell(y_k)\right| &\leq \frac{1}{\epsilon_{i,j}} \left|\text{dist}\left(x_k, \bd T_{c_{\mathcal{H}},\sigma, \psi(b^\ell)}\right) - \text{dist}\left(y_k, \bd T_{c_{\mathcal{H}},\sigma, \psi(b^\ell)}\right)\right|\\
& \leq \frac{1}{\epsilon_{i,j}} \left|x_k - y_k\right|.\end{align*}

Now we claim that the function $\rho_{i,j}^\ell$ is $\frac{1}{\epsilon_{i,j}}$-Lipschitz with respect to the metric on $\mathcal{H}$ given by 
$$ d_\mathcal{H}((x_1,x_2,x_3),(y_1,y_2,y_3)) = |x_1-y_1| + |x_2-y_2| + |x_3 - y_3|.$$

To see this, let $(x_1,x_2,x_3)$ and $(y_1,y_2,y_3)$ be fixed. We compute
\begin{align*}
	&|\rho(x_1,x_2,x_3) -\rho(y_1,y_2,y_3)|\\
	&\leq |f_1(x_1) f_2(x_2) f_3(x_3) - f_1(y_1) f_2(x_2)f_3(x_3)|\\
	&\quad+ |f_1(y_1) f_2(x_2)f_3(x_3) - f_1(y_1)f_2(y_2)f_3(x_3)| \\
	&\quad+ |f_1(y_1)f_2(y_2) f_3(x_3) - f_1(y_1)f_2(y_2) f_3(y_3)| \\
	& \leq \frac{1}{\epsilon_{i,j}}\left(f_2(x_2)f_3(x_3) |x_1 - y_1| + f_1(y_1) f_3(x_3)|x_2-y_2| + f_1(y_1) f_2(y_2) |x_3- y_3|\right)\\
	&\leq \frac{1}{\epsilon_{i,j}}   d_{\mathcal{H}} ((x_1,x_2,x_3) - (y_1,y_2,y_3)). \tag{Since $f_i \leq 1$}
\end{align*}
As in the proof of the uniform comparability of $d_{per}$ and $d_{AGY}$, since $d_{\mathcal{H}}$ comes from a single norm, $d_{\mathcal{H}}$ is uniformly comparable to $d_{per}$.\qed
\end{proof}

We can now use the definition of the $\rho_{i,j}^\ell$ to state a corollary of Theorem~\ref{thm:AGY}.
\begin{coro} \label{cor:AGY_for_rho}
	Fix $0<\delta < 1$. Then there exists a constant $C$ so that for all $j > m_i$, 
	$$\int \rho_{i,j}^i (\rho_{i,j}^j \circ g_{\ell_0(j-i)}) \,d\mu \leq C\mu(A_i)\left[\mu(A_j) + e^{-\frac{\delta}{4}|i-j|}\right].$$
\end{coro}
\begin{proof}
	By definition $\norm{\rho_{i,j}^\ell}_{\infty} = 1$, and by Lemma~\ref{lem:rhoLipschitz} $\norm{\rho_{i,j}^\ell} = \frac{{C'}}{\epsilon_{i,j}} = e^{\frac{\delta}{4}|i-j|}.$ By Theorem~\ref{thm:AGY}, for perhaps a different constant $C$ coming from the uniformly comparable %measures
	{ metrics and} {$C'$,} the fact that $b \geq e$, and writing $C_K +1 = C_K^+$,
\begin{align*}
	&\abs{\int\rho_{i,j}^i (\rho_{i,j}^j \circ g_{\ell_0(j-i)})) - \int \rho_{i,j}^i \int \rho_{i,j}^j} \\
	&\leq C(C_K^+ + e^{\frac{\delta}{4} |i-j|})^2 e^{-\delta|i-j|}\\
	&= C(C_K^+)^2 e^{-\delta|i-j|} + 2CC_K^+ e^{-\frac{3\delta}{4} | i-j|} + Ce^{-\frac{\delta}{2} |i-j|} \tag{Since $e^{-\frac{\delta}{2}|i-j|} \geq e^{-\frac{3}{4}\delta |i-j|} \geq e^{-\delta |i-j|}$ as $e^{-\delta|i-j|} < 1$}\\
	&\leq \tilde{C} e^{-\frac{\delta}{2}|i-j|}.
	\end{align*}
	
	By construction of $\rho_{i,j}^{\ell}$, $\int \rho_{i,j}^\ell = \mu(H_{c_{\mathcal{H}},\sigma,\ell}) = \mu(A_\ell)$, so
	$$\int \rho_{i,j}^i  (\rho_{i,j}^j \circ g_{\ell_0(j-i)})) \leq \int \rho_{i,j}^i \int \rho_{i,j}^j + \tilde{C} e^{-\frac{\delta}{2} |i-j|} \leq \mu(A_i)\mu(A_j) + \tilde{C}e^{-\frac{\delta}{2}|i-j|}.$$
	
	Since $j > m_i$, we have $e^{-\frac{\delta}{4}|i-j|} <\mu(A_i).$ 
	Thus, using $\tilde{C} > 1$,
	$$\int \rho_{i,j}^i  (\rho_{i,j}^j \circ g_{\ell_0(j-i)}))  \leq \tilde{C}\mu(A_i) \left[\mu(A_j) + e^{-\frac{\delta}{4} |i-j|}\right].$$
\qed \end{proof}
Next our goal is to get a relationship between $\mu(A_i\cap A_j)$ and $\int \rho_{i,j}^i (\rho_{i,j}^j \circ g_{\ell_0(j-i)}))$. 

\begin{lem}\label{lem:intersection less than rho product}
 For all $j > m_i$, 
 $$\mu(A_i\cap A_j) \leq \int \rho_{i,j}^i (\rho_{i,j}^j \circ g_{\ell_0(j-i)})) \,d\mu+ \mu(E_j) + \mu(E_i)$$
 where $E_\ell = \{\omega: \rho_{i,j}^\ell \in (0,1)\}.$
\end{lem}
\begin{proof}
	We first make a general claim.
	
	\noindent
	\textbf{Claim} If $0\leq g \leq f\leq 1$ then 
	$$\int f \leq \int g + \mu\{f \neq g\}.$$
	To see this is true, we write 
	$$\int f = \int g+ f - g = \int g + \int_{\{f\neq g\}} f-g  \leq \int g + \mu\{f\neq g\}$$
	where the last inequality follows since $f- g \leq 1$. 
	
	The proof now follows from the claim where $f = \chi_{H_{c_{\mathcal{H}},\sigma, i}\cap g_{\ell_0(j-i)})H_{c_{\mathcal{H}},\sigma,j}}$ and $g = \rho^i_{i,j} (\rho_{i,j}^j\circ g_{\ell_0(j-i)}))$, combined with the fact that $f\neq g$ occurs when $\rho_{i,j}^i \in (0,1)$ or $\rho_{i,j}^j \in (0,1)$. So
	$\mu\{f\neq g\} \leq \mu(E_i) + \mu(E_j).$
\qed \end{proof}

\subsection{Proof that Proposition \ref{prop:Exp_Decay_Borel_Cantelli} Assumption (3) holds} \label{sec:(3)}
To finish the proof of Proposition \ref{prop:Exp_Decay_Borel_Cantelli} Assumption (3) we will relate $\int \rho_{i,j}^i \int \rho_{i,j}^j $ to $\mu(E_i) + \mu(E_j)$ from Lemma~\ref{lem:intersection less than rho product}. In order to do so, we will show it suffices to add a technical assumption about the behavior of $\psi$.

\begin{lem} \label{lem:ugh}
We may assume that for all $i, j$ with  $j> i + \frac{4}{\delta}\log\left(\frac{1}{\mu(A_i)}\right)$
 \begin{equation}\label{eq:ugh}
			e^{-\frac{\delta}{4}|i-j|} \leq\min\left\{ \frac{\sigma^4}{7^4} \frac{1}{\psi(b^j)^4}, \frac{r}{2},  2^{-(n_{\mathcal{H}} +2)}\right\}
		\end{equation}	
		where { $r$ is as in the proof of Lemma~\ref{lem:measurebounds} and }
		$n_{\mathcal{H}} = -\log_2(1-c_{\mathcal{H}})$ for $c_{\mathcal{H}}$ as in Lemma~\ref{lem:measurebounds}.
		\end{lem}

\begin{lem}\label{lem:productrhos_sumE_i}
	Under the assumptions of Lemma~\ref{lem:ugh} there exists a constant $C>1$ so that 
	$$C \int \rho_{i,j}^i {d\mu}\int \rho_{i,j}^j {d\mu}> \mu(E_i) + \mu(E_j).$$
\end{lem}
Indeed Lemma~\ref{lem:productrhos_sumE_i} is sufficient to conclude the proof of (3) as follows:
\begin{theopargself} \begin{proof}[Proposition~\ref{prop:sufficient is satisfied}, part (3)] 
	Fix $i$ and let $j > m_i$. Then
	\begin{align*}
		\mu(A_i\cap A_j) &\leq \int \rho_{i,j}^i (\rho_{i,j}^j \circ g_{\ell_0(j-i)})) + \mu(A_i) + \mu(A_j) \tag{By Lemma~\ref{lem:intersection less than rho product}}\\
		& \leq  \int \rho_{i,j}^i (\rho_{i,j}^j \circ g_{\ell_0(j-i)})) + C \int \rho^i_{i,j} \int \rho^j_{i,j} \tag{By Lemma~\ref{lem:productrhos_sumE_i}}\\
		&\leq C\mu(A_i) \left[\mu(A_j) + e^{-\frac{\delta}{4}|i-j|}\right] + C\mu(A_i)\mu(A_j) \tag{by Corollary~\ref{cor:AGY_for_rho}, and since $\int\rho_{i,j}^\ell \leq \mu(H_{c_{\mathcal{H}},\sigma,\ell}) = \mu(A_\ell)$}\\
		&\leq \tilde{C}\mu(A_i) \left[\mu(A_j) + e^{-\frac{\delta}{4}|i-j|}\right]\tag{Setting $\tilde{C} = 2C$.}
	\end{align*}
\qed \end{proof}\end{theopargself}

Since Lemma \ref{lem:productrhos_sumE_i} depends on the additional assumptions in Lemma~\ref{lem:ugh}, we will first prove Lemma~\ref{lem:ugh} by replacing the function $\psi$ with a function $\tilde{\psi}$ which satisfies Equation~\eqref{eq:ugh} and is still sufficient to prove the desired conclusion for $\psi$. The proof of Lemma~\ref{lem:ugh} uses  Lemma~\ref{lem:psifix}. After stating Lemma~\ref{lem:psifix}, we will then give a proof of Lemma~\ref{lem:ugh}, then Lemma~\ref{lem:productrhos_sumE_i}, and the section concludes with the proof of Lemma~\ref{lem:psifix}.

\noindent

		\begin{lem} \label{lem:psifix}
		 Let $\{a_j\}_{j\in \mathbb{N}}$ be a non-increasing sequence of positive numbers with $\sum_{j=1}^\infty a_j = \infty$. For any $\rho,k,\tau >0$ with $\tau<1$, there exists a constant $C >0$ and a non-increasing sequence $\{c_j\}_{j\in \N}$ with $c_j=\min\left\{\frac 1 j, \max\{a_j, \frac 1 {j^2}\}\right\}$, so that $\sum_{j=1}^\infty c_j = \infty$, $\sum_{j:c_j> a_j}a_j<\infty$, and  

\begin{equation}\label{eq:psifix}
\text{whenever } i\geq 3, \text{ for all } j>\max\{i-C\log(c_i),9\}, \text{ we have } e^{-\rho(j-i)}< \tau c_j^k. 
\end{equation}
		\end{lem} 
		 \begin{theopargself} \begin{proof}[Lemma~\ref{lem:ugh}] 
			We first note that, from the proof of Lemma~\ref{lem:measurebounds}, we can choose $c_{\mathcal{H}}$ close enough to $1$ so that we always have $2^{-n_{\mathcal{H}}} < 2r$, and thus 
			$$\min\left\{ \frac{\sigma^4}{7^4} \frac{1}{\psi(b^j)^4}, \frac{r}{2},  2^{-(n_{\mathcal{H}} +2)}\right\} = \min\left\{ \frac{\sigma^4}{7^4} \frac{1}{\psi(b^j)^4},   2^{-(n_{\mathcal{H}} +2)}\right\}$$		
		 
		 We next construct a modification replacing $\psi$ by $\hat{\psi}$ so that $e^{-\frac{\delta}{4}|i-j|} \leq \frac{\sigma^4}{7^4} \frac{1}{\hat{\psi}}(b^j)^4$ implies $e^{-\frac{\delta}{4}|i-j|} \leq\min\left\{ \frac{\sigma^4}{7^4} \frac{1}{\psi(b^j)^4}, \frac{r}{2},  2^{-(n_{\mathcal{H}} +2)}\right\}$. 
		 Namely using the constants from Lemma~\ref{lem:measurebounds}, define
		 $$\widehat{\psi}(b^i) = \begin{cases}\psi(b^i) & \frac{M \sigma^2}{\psi(b^i)} <  2^{-(n_{\mathcal{H}} +2)} \\ M \sigma^2 2^{(n_{\mathcal{H}} +2)} & \text{otherwise.}\end{cases}$$
		We now have the upper bounds of $r/2$ and $2^{-n_\mathcal{H}+2}$ are trivially satisfied for $\widehat{\psi}$ for all $j > i  + \frac{4}{\delta} \log\left(\frac{1}{\mu(A_i)}\right)$. Moreover for sets $\widehat{A}_j$ corresponding to $\widehat{\psi}$, $\widehat{A}_j \subseteq A_j$, and so $\mu(\limsup A_j) \geq \mu(\limsup \widehat{A}_j)$. Since our goal is to prove positive measure, we may now always assume the upper bounds of $r/2$ and $2^{-n_\mathcal{H}+2}$ hold.
		
		We now construct $\widetilde{\psi}$ from $\psi$ which satisfies Equation~\eqref{eq:ugh} by defining $\widetilde{\psi}(b^j) = c_j^{-1/2}$ where $c_j = \min\{1/j, \max\{a_j, 1/j^2\}\}$ is the sequence from Lemma~\ref{lem:psifix} for $a_j = \frac{1}{\psi(b^j)^2}$, $\rho =\frac{\delta}{4}$, $k=2$, and $\tau = \frac{\sigma^4}{7^4} <1$. Let $\widetilde{A}_i$ be the set corresponding to $\widetilde{\psi}$.
		
		Notice that $c_i = \frac{1}{\tilde{\psi}(2^i)^2}$, so by Lemma~\ref{lem:psifix} whenever $i\geq 3$ and $j > \max\{i- C\log(\tilde{\psi}(2^i)^2),9\}$, we indeed have for $i$ large enough that making $C$ larger if necessary
		$$e^{-\frac{\delta}{4}|i-j|} < (\widetilde{\psi}(2^i))^{-2 \cdot C\frac{\delta}{4}} < m \sigma^2(\widetilde{\psi}(2^i))^{-2 } \leq \mu(\widetilde{A}_i) $$
		 and moreover we have the desired inequality that
 $$e^{-\frac{\delta}{4} |i-j|} < \frac{\sigma^4}{7^4}\frac{1}{\tilde{\psi}(2^j)^4} .$$
 	
 		The restrictions on $i\geq 3$ and $j >9$ do not play a role in the conclusion for the limsup sets. For $j \geq n > 9$, notice $c_j \leq a_j$ implies $\widetilde{A}_j \subseteq A_j$, so 
 		$$\bigcup_{j=n}^\infty A_j \supseteq \bigcup_{\stackrel{j=n}{c_j\leq a_j}} \tilde{A}_j \cup \bigcup_{\stackrel{j=n}{c_j>a_j}} {A}_j \supseteq \bigcup_{\stackrel{j=n}{c_j\leq a_j}} \tilde{A}_j .$$
 		On the other hand since $c_j > a_j$ exactly when $c_j = \frac{1}{j^2}$,
$$\mu\left(\bigcup_{j=n}^\infty \tilde{A}_j\right) \leq \mu\left( \bigcup_{\stackrel{j=n}{c_j \leq a_j}}^\infty \tilde{A}_j\right)  + \mu \left(\bigcup_{\stackrel{j=n}{c_j > a_j}}^\infty \tilde{A}_j\right) \leq \mu\left( \bigcup_{\stackrel{j=n}{c_j \leq a_j}}^\infty \tilde{A}_j\right)  + \sum_{\stackrel{j=n}{c_j>a_j}}^\infty \frac{m \sigma^2}{j^2} .$$
	Since the tail sum of convergent series goes to zero, for any $\epsilon$ we can choose $n$ large enough  so that $\sum_{j>n, c_j>a_j} \frac{m\sigma^2}{j^2} < \epsilon$,
Hence,
$$\mu\left(\bigcup_{j=n}^\infty A_j\right) \geq \mu\left(\bigcup_{\stackrel{j=n}{c_j\leq a_j}} \tilde{A}_j\right) \geq \mu\left(\bigcup_{j=n}^\infty \tilde{A}_j\right) - \epsilon.$$
Thus in the limit we conclude $\mu(\limsup A_j) \geq \mu(\limsup \tilde{A}_j).$
\qed \end{proof}\end{theopargself}
 With our preliminary work, Lemma~\ref{lem:productrhos_sumE_i} follows from geometric estimates on measures of balls and trapezoids.

\begin{theopargself} \begin{proof}[Lemma \ref{lem:productrhos_sumE_i}]
	Since $\rho_{i,j}^j \leq \rho_{i,j}^j$, we have 
	\begin{equation} \label{eq:rhoj=1}
	\int \rho_{i,j}^i \int \rho_{i,j}^j \geq \left(\int \rho_{i,j}^j\right)^2 \geq \mu\{\rho_{i,j}^j = 1\}^2.
	\end{equation}
	Now we want to get a lower bound for the measure of the set where $\rho_{i,j}^j = 1$. To do this, we need to find the area of the subset of the trapezoid $T_{c_{\mathcal{H}},\sigma,j}^+$ where the $\rho_{i,j}^j = 1$. To do this, note the horizontal line $y= \epsilon_{i,j}$ from $c_{\mathcal{H}}+\epsilon$ to $1-\epsilon$ give the height of the inner trapezoid. The line which is length $\epsilon_{i,j}$ away from the line $y= \frac{\sigma}{\psi(b^j)} x$ is given by 
	$$y = \frac{\sigma}{\psi(b^j)} x - \epsilon\sqrt{1 + \frac{\sigma^2}{\psi(b^j)^2}}.$$
	Thus the four corners of the trapezoid where $\rho_{i,j}^j = 1$ are given by $(c_{\mathcal{H}}+\epsilon )+i \epsilon, (1-\epsilon) + i \epsilon,$
	$$\left(1-\epsilon\right) + i\left( \frac{\sigma}{\psi(b^j)}(1-\epsilon) - \epsilon\sqrt{1+\frac{\sigma^2}{\psi(b^j)^2}}\right),$$ and $$\left(c_{\mathcal{H}}+\epsilon\right) + i\left(\frac{\sigma}{\psi(b^j)}(c_{\mathcal{H}}+\epsilon) - \epsilon\sqrt{1+\frac{\sigma^2}{\psi(b^j)^2}}\right). $$
	
	Hence the area of subset of $T_{c_{\mathcal{H}},\sigma,j}^+$ where $\rho_{i,j}^j = 1$ is given by 
	$$\frac{(1-c_{\mathcal{H}}-2\epsilon_{i,j})}{2}\left(\frac{\sigma}{\psi(b^j)}(1+c_{\mathcal{H}}) -2\epsilon_{i,j}\sqrt{1 + \frac{\sigma^2}{\psi(b^j)^2}}-2\epsilon_{i,j}\right). $$
By symmetry, the area of $T_{c_{\mathcal{H}},\sigma,j}^-$ where $\rho_{i,j}^j = 1$ is the same as the area for $T_{c_{\mathcal{H}},\sigma,j}^+$. Thus the total area is the product of the two trapezoids with the product of the ball where $n + 4 = 2g + s -1$ and $\sigma(n)$ is gives the volume of the $n$-ball. That is 
	\begin{align*}
	\mu(\{\rho_{i,j}^j= 1\}) &= \frac{(1-c_{\mathcal{H}}-2\epsilon_{i,j})^2\sigma(n)(r-\epsilon_{i,j})^n}{4}\\
	 &\hspace{.25in}\cdot \left(\frac{\sigma}{\psi(b^j)}(1+c_{\mathcal{H}}) -2\epsilon_{i,j}\sqrt{1 + \frac{\sigma^2}{\psi(b^j)^2}}-2\epsilon_{i,j}\right)^2\\
	&\geq d_n(1-c_{\mathcal{H}}-2\epsilon_{i,j})^2\left(\frac{\sigma}{\psi(b^j)} -2\epsilon_{i,j}\sqrt{1 + \frac{\sigma^2}{\psi(b^j)^2}}-2\epsilon_{i,j}\right)^2 \tag{by \eqref{eq:ugh}, $\epsilon_{i,j} \leq \frac{r}{2}$ so $d_n$ depends only on $n$, and since $1+ c_{\mathcal{H}} \geq 1$}\\
	&= d_n \left(2^{-n_{\mathcal{H}}}- 2\epsilon_{i,j}\right)^2\left[\frac{\sigma}{\psi(b^j)} - 6\epsilon_{i,j}\right]^2\tag{substituting $1-c_\mathcal{H} = 2^{-n_\mathcal{H}}$ from Lemma~\ref{lem:measurebounds} and $\sigma^2 / \psi(b^j)^2 < 1 < 3$} \\
	& \geq \frac{d_n}{2^{2(n_\mathcal{H} + 1)} }\left[\frac{\sigma}{\psi(b^j)} - 6\epsilon_{i,j}\right]^2 \tag{assuming by \eqref{eq:ugh} $\epsilon_{i,j} < 2^{-(n_\mathcal{H} + 2)}$, which  implies $2^{-n_\mathcal{H}} - 2\epsilon_{i,j} > 2^{-(n_\mathcal{H} + 1)}$ }\\
	& \geq  \frac{d_n}{2^{2(n_\mathcal{H} + 1)}} \epsilon_{i,j}^\frac{1}{2}. \tag{assuming by \eqref{eq:ugh} that $\epsilon_{i,j} \leq \frac{\sigma^4}{7^4 \psi(b^j)^4}$, thus $6\epsilon_{i,j} + \epsilon_{i,j}^\frac{1}{4} \leq 7 \epsilon_{i,j}^\frac{1}{4}\leq \frac{\sigma}{\psi(b^j)}$}
	\end{align*}
	Combining this fact with \eqref{eq:rhoj=1}, we obtain
	\begin{equation}\label{eq:lowerboundproduct}
	 \int \rho_{i,j}^i \int \rho_{i,j}^j \geq \frac{d_n^2}{2^{4(n_\mathcal{H} + 1)}} \epsilon_{i,j}. 
	\end{equation}
	
	Now from the other end we want an upper bound for $\mu\{\rho^i \in (0,1)\} + \mu\{\rho^j \in (0,1)\} \leq 2C\epsilon_{i,j}.$

	Given $\ell = i$ or $\ell = j$, we have the area of the $\epsilon_{i,j}$-boundary of one of the trapezoids $T_{c_\mathcal{H},\sigma, 2^\ell}^\pm$ is given by 
 	\begin{align*}
 		&\mu(\bd_{\epsilon_{i,j}}T_{c_\mathcal{H},\sigma,2^\ell}^\pm) \\
 		&=\frac{\sigma}{2 \psi(b^\ell)} (1-c_\mathcal{H}^2) - \left(\frac{1-c_\mathcal{H}}{2} - \epsilon_{i,j}\right)\left(\frac{\sigma}{\psi(b^\ell)}(1+c_\mathcal{H})-2\epsilon_{i,j} \sqrt{1 + \frac{\sigma^2}{\psi(b^\ell)^2}}\right)\tag{taking area of $T_{c_\mathcal{H},\sigma,\ell}^\pm$ less the area where $\rho_{i,j}^\ell = 1$}\\
 		&=  \epsilon_{i,j}\left[\frac{\sigma}{\psi(b^\ell)}(1+c_\mathcal{H}) + \left(1-c_\mathcal{H} - 2\epsilon_{i,j}\right)\left[1 + \sqrt{1 + \frac{\sigma^2}{\psi(b^\ell)^2}}\right]\right]\\
 		&\leq \epsilon_{i,j}\left[1+c_\mathcal{H} + (1-c_\mathcal{H}) (1 + \sqrt{2})\right]  \tag{assuming $\sigma < \psi(b^\ell)$ which is easy since $\sigma < 1$ is fixed and $\psi \geq 1$ is nondecreasing} \\
 		&\leq \epsilon_{i,j} C \tag{where $C$ depends on $c_\mathcal{H}$}.
 	\end{align*}
 	Thus
 	\begin{equation} \label{eq:sumupperbound}
 		\mu(\{\rho^i_{i,j} \in (0,1)\}) + \mu(\{\rho^j_{i,j}\in (0,1)\}) \leq 2 C\epsilon_{i,j}.
 	\end{equation}
	Combining Equation~\ref{eq:sumupperbound} with Equation~\ref{eq:lowerboundproduct}, 
	\begin{align}
		 &\int \rho_{i,j}^i \int \rho_{i,j}^j \geq \frac{d_n^2}{2^{4(n_\mathcal{H} + 1)}} \epsilon_{i,j}\nonumber\\
		 & \geq \frac{d_n^2}{2^{4(n_\mathcal{H} + 1)} (2C)} \left[\mu\{\rho^i_{i,j}\in (0,1)\} + \mu\{\rho^j_{i,j} \in (0,1)\}\right].\label{eq:eps_jlowerbound}
	\end{align}
	
	Setting $\tilde{C} = \frac{2^{4(n_\mathcal{H} + 1)}(2C)}{d_n^2}$, we can assume $\tilde{C}>1$ since $d_n$ is bounded above by a fixed constant, and we can make $C$ larger if necessary.
\qed \end{proof}\end{theopargself}

\begin{theopargself} \begin{proof}[Lemma~\ref{lem:psifix}.]
We first claim the following:

\noindent{\textbf{Claim}:} The sequence $c_j$ is non-increasing, has divergent sum and 
$$\sum_{j:c_j>a_j}a_j<\infty, \quad \sum_{j:c_j> a_j}c_j<\infty.$$
\begin{theopargself} \begin{proof}[of Claim.] The maximum of two non-increasing sequences is non-increasing and so $\max\{a_j, \frac 1 {j^2}\}$ is a non-increasing sequence (in $j$). 
Similarly the minimum of two non-increasing sequences is non-increasing and so $c_j$ is non-increasing. 

If $\max\{a_j, \frac{1}{j^2}\} = \frac{1}{j^2}$, then $c_j = j^{-2} > a_j$. Otherwise $c_j = \min\{1/j, a_j\} \leq a_j.$

If $c_j> a_j$ then $a_j<\frac 1 {j^2}$ and so clearly $\sum_{c_j>a_j}a_j<\sum \frac 1 {j^2}<\infty$. 
On the other hand, since $c_j>a_j$ is only possible when $c_j = j^{-2}$,
$$\sum_{c_j> a_j} c_j = \sum_{c_j >a_j} j^{-2} <\infty.$$

Now observe that 
\begin{equation}\label{eq:cauchy cond}\sum_{j=2^{k}}^{2^{k+1}-1}c_j\geq \min\left\{\frac 1 2 ,2^ka_{2^{k+1}}\right\}.
\end{equation} Indeed we are estimating the sum from below by $2^kc_{2^{k+1}}$ and considering the different possibilities of $c_{2^{k+1}}$. Notice that $2\sum_k 2^ka_{2^{k+1}}\geq \sum_ja_j$ and so $\sum2^ka_{2^{k+1}}$ diverges and thus $\sum_k \min\{\frac 1 2 ,2^ka_{2^{k+1}}\}$ diverges. So $\sum_j c_j=\sum_k \sum_{j=2^k}^{2^{k+1}-1}c_j$ diverges. 
\qed \end{proof}\end{theopargself}

We have proved the Claim and now proceed with the remainder of the proof of Lemma~\ref{lem:psifix}. We now show if $C>4\frac {\tilde{k}} \rho(\frac{1}{\tilde{\tau}}+1)$
then Equation~\eqref{eq:psifix} holds where $\tilde{\tau}=\min\{\tau,1\}$ and $\tilde{k}=\max\{\frac 1 2 ,k\}$. Indeed it suffices to show that for all $j>i+C\log(i)$ we have that 
\begin{equation}\label{eq:5.1 bound}
e^{-\rho(j-i)}<\frac{\tau}{ {j}^{2k}}.
\end{equation} 
 Clearly smaller $\tau$ and larger $2k$ make \eqref{eq:5.1 bound} harder to satisfy. So from here we assume $\tau\leq 1$ and $2k\geq 1$, which motivated our choice of $\tilde{\tau}$ and $\tilde{k}$. Under these assumptions, Equation \eqref{eq:5.1 bound} is implied by $j-i >\frac {2k} {\rho}\log(\frac{j}{\tau})$.

If $j\leq 2i$ this follows from our condition on $C$. Indeed 
$$j-i>4 \frac k \rho\left(\frac 1 \tau+1\right) \log(i)> 4\frac k \rho\left(\frac 1 \tau+1\right)(\log(j)-\log(2)).$$ And so 
$$j-i>4\frac{k}\rho(\log\left({j}^{\frac 1 \tau+1}\right)-\log(2^{\frac 1 \tau+1}))> 2\frac k \rho \log(j^{\frac 1 \tau+1})> \frac {2k} {\rho}\log\left(\frac{j}{\tau}\right).$$  Note that the second inequality uses that $\log(j^{\frac 1 \tau+1})>2\log(2^{\frac 1 \tau+1})$ because $j\geq 9>2^2$ and the third inequality uses that $j^{\frac 1 \tau+1}>\frac 1 \tau j$ for all $j\geq 9$ and $\tau>0$.

For the case when $j>2i$, set $f(x)=x-i$ and $g(x)=\frac  {2k}{\rho} \log\left(\frac{x}{\tau}\right)$. Note that $f(2i) > g(2 i)$ from the case where $j\leq 2i$. Moreover, $f'(x)=1>g'(x)=\frac {2k}{x\rho\tau}$ for all $x>\frac{2k}{\rho\tau}.$ Since $i\geq 3$ and $\log(3) > 1$,
$$j  > i+ C\log(i) > 3 + \frac{4k}{\rho} (\frac1\tau+1)  \log(3) > \frac{2k}{\rho\tau}.$$ 
So for all $j> 2 i$ we have $f'(j) > g'(j)$ and $f(2  i) > g(2 i)$. Hence $f(j) > g(j)$ for all $j\geq 2 i$ as desired.  
\qed \end{proof}\end{theopargself}

\section{Verifying Proposition~\ref{prop:Exp_Decay_Borel_Cantelli} Assumption (4)}\label{sec:averages}
We begin this section by defining the sets $B_i$ and $C_i$ required for Proposition~\ref{prop:sufficient is satisfied}, and then verify these sets satisfy Assumption (4) of 
Proposition~\ref{prop:Exp_Decay_Borel_Cantelli}.
\begin{defn}[Definition of the $B$'s] Set $I \stackrel{def}{=} (-\frac{\pi}{12}, \frac{\pi}{12}).$ For $k\in \N$ define
	$$B_k = g_{\log(b^k)} g_{\log\left(\sqrt{\frac{\psi(b^k)}{\sigma}}\right)}\bigcup_{\theta \in I} r_\theta W_k$$
	where we have the following definitions for the sets shown in Figure~\ref{fig:1}.
	We first pull back the set $A_k$ so that trapezoids in $H_{c,\sigma,k}$ makes a 45 degree angle and define
	$$\widetilde{W}_k = g_{log\left(\sqrt{\frac{\sigma}{\psi(b^k)}}\right)} g_{-log(b^k)} A_k.$$
	Then we restrict to a smaller subset of $\widetilde{W}_k$ denoted $W_k$ so that $r_\theta W_k \subset \widetilde{W}_k$ for $\theta \in I$. That is $W_k$ is the set of $\omega$ with two holonomy vectors $v_1$ and $v_2$ satisfying 
	$$ c_{\mathcal{H}}\sqrt{\frac{2\sigma}{\psi(b^k)}} \leq |v_1|, |v_2| \leq \sqrt{\frac{\sigma}{\psi(b^k)}} ,$$
	and  $$ \textnormal{arg}(v_1) \in \left(\frac{\pi}{12}, \frac{\pi}{6}\right), \quad \textnormal{arg}(v_2) \in \left(-\frac{\pi}{6}, -\frac{\pi}{12}\right).$$
\begin{figure}
	\begin{center}
		\begin{tikzpicture}[scale=3]
	\filldraw[gray!40!white, draw=black] (1/2,1/2)--(1,1)--(1,-1)-- (1/2, -1/2)--cycle;	
	\draw(0,0)-- (1,1);
	\draw (0,0)-- (1,-1);
	\draw (0,0)-- (1,0);
	\draw (1,-1)-- (1,1);
	\draw[red!30!white, fill = red!30!white] (.68,.28* .68) arc (15:30:1/1.41)  --(.86,.5) arc(30:15:1)  --cycle;
	\draw[red!30!white, fill = red!30!white] (.68,.28* -.68) arc (-15:-30:1/1.41)  --(.86,-.5) arc(-30:-15:1)  --cycle;
	\draw[red!60!white, dashed](1/1.41,-1/1.41) arc (315:405:1);
	\draw[red!60!white, dashed](1/2,-1/2) arc (315:405:1/1.41);
	\draw[red!60!white] (.68,.28* .68) -- (.96,.27);
	\draw[red!60!white] (.68, -.28*.68 ) -- (.96,-.27); 
	\draw[red!60!white] (.62,.58 * .62) -- (.86,.5);
	\draw[red!60!white] (.62,-.58*.62) -- (.86,-.5); 
\end{tikzpicture}	
\end{center}
\caption{The gray shaded region corresponds to the region $\tilde{W}_k$ and the pink region corresponds to $W_k$.
}
\label{fig:1}   
\end{figure}
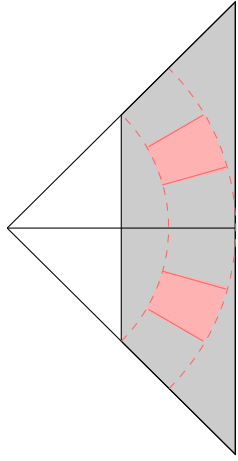
\end{defn}
\begin{defn}[Definition of the $C$'s]
	Define 
	$$C_k = g_{\log(b^k)} g_{\log\left(\sqrt{\frac{\psi(b^k)}{\sigma}}\right)}S(c_{\mathcal{H}},\sigma, b^k)$$
	where 
	\begin{align*} &S(c_{\mathcal{H}},\sigma, t) = \left\{\omega\in \mathcal{H}: \substack{\omega \text{ has a holonomy vector } v \text{ with }\\|v| \in \left(c_{\mathcal{H}}\sqrt{\frac{\sigma}{\psi(t)}}, \sqrt{\frac{2\sigma}{\psi(t)}}\right)} \right\}.
	\end{align*}
	
\end{defn}

\subsection{Proof that (4) (a)-(c) hold}\label{subsec:4a-c} 
We now verify assumptions (4) (a), (b) and (c), where we set $c = c_{\mathcal{H}}$.
Fix $i$ and take $j$ so that $i<j\leq i + C\log\left(\frac{1}{\mu(A_i)}\right)$. 
	\begin{enumerate}
		\item[(a)] As constructed $B_i \subseteq A_i$ and $A_j \subseteq C_j$.
		\item[(b)] Following the proof of Lemma~\ref{lem:measurebounds} where in Equation~\ref{eq:Balldef}, we replace the trapezoids by sectors of annuli, we obtain a constants $m',m>0$ so that
		$$\mu(W_i) \geq m' \frac{\sigma^2}{\psi(b^i)^2}\left(\frac{\pi}{24} \left(1 - 2c_{\mathcal{H}}\right)\right)^2 {{\textbf{m}}}_{2g+s-3}(B) > m \frac{\sigma^2}{\psi(b^i)^2}.$$ Thus by Lemma~\ref{lem:measurebounds}, $\mu(B_i) \geq \mu(W_i) \geq m \mu(A_i) /M.$
		\item[(c)] Since the measure is invariant under geodesic flow and by Lemma~\ref{lem:measurebounds},
		$$\mu(A_j) = \mu(H_{c_{\mathcal{H}},\sigma, j}) \geq m \frac{\sigma^2}{\psi(b^j)^2}.$$
		By Masur--Smillie Lemma~\ref{lem:MasurSmillie},
		$$\mu(C_j) = \mu(S(c_{\mathcal{H}},\sigma,2^j)) \leq M \frac{\sigma}{\psi(b^j)}.$$
		Thus
		$$\mu(C_j) \leq \frac{M}{\sqrt{m}} \sqrt{\frac{m\sigma^2}{\psi(b^j)^2}} = \frac{M}{\sqrt{m}} \mu(A_j)^\frac{1}{2}.$$
\end{enumerate}

\subsection{Construction and circle averages of logsmooth functions}\label{sec:circleaverages}

	The main goal of this section is to prove Corollary~\ref{cor:AverageUpperBound}, which extends the statements of \cite{Dozier19} (giving averages over intervals) to include so-called logsmooth functions from \cite{Athreya06Thesis} (which gives averages over the full circle) 
	\begin{defn}
		A \textit{complex} $K$ in $\omega$ is a closed subset of $X$ whose boundary $\bd K$ consists of a union of disjoint (in the interior) saddle connections such that if $\bd K$ contains three saddle connections bounding a triangle, then the interior of that triangle is in $K$. Given a complex $K$ the \textit{complexity} of $K$ is the number of saddle connections needed to triangulate $K$ {and $|\bd K|$ is the length of the longest saddle connection in $\bd K$.}  For any $\delta > 0$ and $k\in \N$, if $M$ is the complexity of $\omega$,
		$$\alpha_k(\omega) = \max_{\stackrel{K \text{ complexity } k}{\text{area}(K) < 2^{k-M-1}}} \frac{1}{|\bd K|^{1+\delta}}.$$
		If the set over which we take the maximum is empty, then we set $\alpha_k(\omega) = 0$. 
		\end{defn} 
		\begin{defn}
			Given a function $f$ on $\mathcal{H}$ and a point $\omega \in \mathcal{H}$, we let 
			$$\text{Ave}_t(f)(\omega) = \frac{1}{2\pi} \int_0^{2\pi} f(g_tr_\theta \omega) \, d\theta.$$
		\end{defn}
	
	Note that $\alpha_1(\omega) = \frac{1}{\ell(\omega)^{1+\delta}}$ where $\ell(\omega)$ is the length of the shortest saddle connection. Since $M$ is finite for all $k$ large enough, $\alpha_k(\omega) = 0$. For more information and intuition for the $\alpha_k$, see Section~5.3 of \cite{Dozier19}. From Proposition~5.{5} of \cite{Dozier19}, we have
	\begin{prop}\label{Prop5.3Dozier}
		Fix a stratum $\mathcal{H}$, and $0<\delta < \frac{1}{2}$. We can find a constant $b$ such that for any interval $I\subseteq S^1$, there exists a constant $c_I$ such that for all $\omega \in \mathcal{H}$ and $T\geq 0$,
		$$\int_I \alpha_k(g_Tr_\theta\omega) \,d\theta < c_I e^{-(1-2\delta)T} \sum_{j\geq k} \alpha_j(\omega) + b |I|.$$
	\end{prop}
	
	The strategy we will take is to extend this theorem to a function $V_\delta$ which is a weighted average of the $\alpha_k$ functions. {The choice of weights for the $\alpha_k$ ensures the functions $V_\delta$ have the following properties stated in Lemma~\ref{lem:JayadevsThesis}.} 
	
	%We want to weight the average to have nice properties, so our goal is to recreate the following theorem for integrating over an interval $I$ instead of $[0,2\pi)$.

	\begin{lem}[Lemma~6.2 \cite{AvilaGouezel13}, Proof in \cite{Athreya06Thesis}] \label{lem:JayadevsThesis}
 Let $\mathcal{V}$ be a neighborhood of the identity in $\mathrm{SL}(,\R)$. Fix $\mathcal{H}$ a connected component of a stratum. For every $0<\delta <1$ there exists $c_1>0$ so that for all $t>0$ there exists a function $V_{\delta}^{(t)}: \mathcal{H} \to [1,\infty)$ and a scalar $b_t$ satisfying the following properties. 
 For all $\omega \in \mathcal{H}$, 
 $$\text{Ave}_{t}(V_{\delta}^{(t)})(\omega) = \int_0^{2\pi} V_{\delta}^{(t)}(g_t r_\theta \omega)\,d\theta \leq c_1e^{-(1-\delta)t} V_{\delta}^{(t)}(\omega) + b_t.$$

Moreover, $V_\delta$ is logsmooth. That is
\begin{equation} \label{eq:logsmooth}
V_{\delta}^{(t)}(g\omega)\leq c_3 V_{\delta}^{(t)}(\omega) 
\end{equation}
for all $\omega \in \mathcal{H}$ and $g\in \mathcal{V}$.

Finally, there exists a constant $C_{\delta,t}$ so that 
\begin{equation}
	\label{eq:v_deltacomparabletoalpha1}
	\frac{V_{\delta}^{(t)}(\omega)}{V_\delta (\omega)} \in [C_{\delta, t}^{-1}, C_{\delta, t}]
\end{equation}
where $V_{\delta} = \max\{1, \alpha_1(\omega)\} = \max\{1, \frac{1}{\ell(\omega)^{1+\delta}}\}$.
\end{lem}

{We will build up to Corollary~\ref{cor:AverageUpperBound}, which is similar to the conclusion of Lemma~\ref{lem:JayadevsThesis} where we average over an interval $I$ instead of $[0,2\pi)$.} We now explicitly construct $V_\delta$ using the following result.
\begin{prop}[Proposition~5.6 \cite{Dozier19}]
	Fix $\mathcal{H}$ and $0<\delta <1$. There exists $C >0$ so that for any $t>0$, there exists constants $b_t$ and $w_t$ so that for any $k$ and any $\omega \in \mathcal{H}$,
	\begin{equation}\label{eq:Dozier5.4}
	\text{Ave}_t(\alpha_k)(\omega) \leq C e^{-t(1-\delta)} \alpha_k(\omega) + w_t \sum_{j>k} \alpha_j(\omega) + b_t.
	\end{equation}
\end{prop}
	
	\begin{defn}
		Fix $\delta$ and $t>0$. Define 
		$$\lambda_k^{(t)} = \left(\frac{w_t}{C} + 1\right)^k$$
		where $w_t$ and $C$ are the constants of \eqref{eq:Dozier5.4}. Define $$V_\delta^{(t)}(\omega) = \sum_{k=0}^M \lambda_k^{(t)} \alpha_k(\omega)$$
		where $M$ is the maximum complexity of $\omega$. 
	\end{defn}
	\begin{theopargself} \begin{proof}[of Lemma~\ref{lem:JayadevsThesis}.]
		We first claim 
	\begin{equation}\label{eq:lambda bounds}
		\lambda_k^{{(t)}} Ce^{-(1-\delta)t} + w_t \sum_{j=0}^{k-1} \lambda_j^{{(t)}} \leq 2C \lambda_k^{{(t)}} e^{-(1-\delta)t}.
	\end{equation}
	
	To see this holds, note that $e^{-(1-\delta)t} \geq 1$. Thus since $\lambda_k^{{(t)}} \geq 1$, we have 
	$$1 - \frac{1}{\lambda_k^{{(t)}}} \leq 1 \leq \left(\frac{w_t}{C} +1 - 1\right) \frac{C}{w_t} e^{-(1-\delta)t}.$$
	Simplifying and using the finite geometric series formula, this implies
	$$\sum_{j=0}^{k-1} \lambda_j^{{(t)}} = \frac{\lambda_k^{{(t)}} - 1}{\lambda_1^{{(t)}} - 1} \leq \lambda_k^{{(t)}} \frac{C}{w_t}e^{-(1-\delta)t}.$$
	Multiplying by $w_t$ and adding $\lambda_k^{{(t)}} C e^{-(1-\delta)t}$ to each side yields \eqref{eq:lambda bounds}. 
	
	 We now want to prove that on average $V_{\delta}^{{(t)}}$ shrinks over circles of radius $t$. To see this, we compute
	\begin{align*}
		&\text{Ave}_t(V_{\delta}^{{(t)}})(\omega) \\
		&= \sum_{k=0}^{{M}} \lambda_k^{{(t)}} \text{Ave}_t(\alpha_k)(\omega)\\
		&\leq \sum_{k=0}^{{M}} \lambda_k^{{(t)}} \left(Ce^{-t(1-\delta)} \alpha_k(\omega) + w_t \sum_{j>k} \alpha_j(\omega) + b_t \right) \tag{by \eqref{eq:Dozier5.4}}\\
		&= \left(\sum_{k=0}^{{M}} \lambda_k^{{(t)}} Ce^{-t(1-\delta)} \alpha_k(\omega)\right) + w_t \sum_{k=1}^{{M}} \alpha_k(\omega) \left(\sum_{j=0}^{k-1} \lambda_j^{{(t)}}\right) + b_t \tag{replacing $b_t = b_t \left(\sum_{k=0}^{{M}} \lambda_k^{{(t)}}\right)$}\\
		&\leq 2Ce^{-t(1-\delta)} \lambda_0 ^{{(t)}} \alpha_0(\omega) + \sum_{k=1}^{{M}} \alpha_k(\omega) \left[\lambda_k^{{(t)}} Ce^{-t(1-\delta)} + w_t \sum_{j=0}^{k-1} \lambda_j\right] + b_t \tag{by \eqref{eq:lambda bounds} and replacing $C$ with $2C$}\\
		&\leq C e^{-t(1-\delta)} \sum_{k=0}^{{M}} \lambda_k^{{(t)}} \alpha_k(\omega) + b_t \\
		&=  Ce^{-t(1-\delta)} V_\delta^{{(t)}}(\omega) + b_t.
	\end{align*}

%	\begin{align*}
%		&\text{Ave}_t(V_{\delta})(\omega) \\
%		&= \sum_{k=0}^n \lambda_k \text{Ave}_t(\alpha_k)(\omega)\\
%		&\leq \sum_{k=0}^n \lambda_k \left(Ce^{-t(1-\delta)} \alpha_k(\omega) + w_t \sum_{j>k} \alpha_j(\omega) + b_t \right) \tag{by \eqref{eq:Dozier5.4}}\\
%		&= \sum_{k=0}^n \lambda_k Ce^{-t(1-\delta)} \alpha_k(\omega) + w_t \sum_{k=1}^n \alpha_k(\omega) \left(\sum{j=0}^{k-1} \lambda_j\right) + b_t \tag{replacing $b_t = b_t \left(\sum_{j=1}^n \lambda_j\right)$}\\
%		&\leq 2Ce^{-t(1-\delta)} \lambda_0 \alpha_0(\omega) + \sum_{k=1}^n \alpha_k(\omega) \left[\lambda_k Ce^{-t(1-\delta)} + w_t \sum_{j=0}^{k-1} \lambda_j\right] + b_t \tag{by \eqref{eq:lambda bounds} and replacing $C$ with $2C$}\\
%		&\leq C e^{-t(1-\delta)} \sum_{k=0}^n \lambda_k \alpha_k(\omega) + b_t \\
%		&=  Ce^{-t(1-\delta)} V_\delta(\omega) + b_t.
%	\end{align*}
	
	The logsmoothness of the $V_{\delta}^{{(t)}}$ follows from \cite[Proof of Proposition 7.2]{EskinMasur01}.
	\qed \end{proof}\end{theopargself}

Now that we have defined $V_\delta^{{(t)}}$ with the logsmooth property, we proceed to extending the results of \cite{Dozier19} to include the $V_\delta^{{(t)}}$ function.

\begin{lem}\label{lem:DozierShadowing} 

There exists a constant $c_2>0$ so that for any $\tau > 0$ and $I\subseteq S^1$ an interval, there exists $t_0(\tau,|I|) \geq 0$ so that for any $\omega \in \mathcal{H}$ and $t> t_0$, we have 
$$\int_I V_{\delta}^{(\tau)}(g_{t+\tau} r_\theta \omega) \,d\theta \leq c_2 \int_{J} \text{Ave}_\tau(V_\delta^{(\tau)})(g_tr_\theta \omega)\,d\theta$$
where $J\subseteq S^1$ is an interval (that could depend on all other parameters) with $|J| = |I|$. 
\end{lem}

\begin{proof}
	Note that this result would follow directly from linearity combined with Lemma~5.2 of \cite{Dozier19}, except as stated in Lemma~5.2 the interval $J$ could depend on $\alpha_i$. However following the proof exactly using linearity to replace each $\alpha_i$ with $V_\delta^{(\tau)}$, we take the interval $2I$ with the same center as $I$ and twice the length. Then in the last 5 lines of the proof, we write $2I = J_1 \cup J_2$ as a union of two intervals with $|J_1| = |J_2| = |I|$. Then 
	$$\max_{j=1,2} \int_{J_j}  \text{Ave}_\tau(V_\delta^{(\tau)}(g_tr_\theta \omega) \geq \frac{1}{2} \int_{2I} \text{Ave}_{\tau}V_{\delta}^{(\tau)} (g_tr_\theta \omega).$$
	Now define $J$ (which now depends on $V_\delta^{{(}\tau{)}}$ instead of individual $\alpha_k$) to be the interval on which the maximum is achieved, and the proof follows by linearity as desired.
\qed \end{proof}

Now we state Proposition~5.{5} of \cite{Dozier19} for the $V_\delta^{{(t)}}$ functions.

\begin{prop} \label{prop:VdeltaAverageInterval}
	Fix a stratum $\mathcal{H}$ and $0<\delta < 1$. Let $c_1$ and $c_2$ be the constants of Lemma~\ref{lem:JayadevsThesis} and Lemma~\ref{lem:DozierShadowing}, respectively. Choose $\tau> 0$ large enough so that 
	$$c_1c_2e^{-(1-\delta)\tau} < \frac{1}{2}.$$ Let $I\subseteq S^1$ be an interval and by Lemma~\ref{lem:DozierShadowing} let $m$ be the smallest possible integer so that $(m-1)\tau > t_0(\tau,|I|).$ That is $m = 1 + \left \lceil \frac{t_0(\tau, |I|)}{\tau}\right\rceil$.
	There are constants $c= c(\tau,\delta, |I|) >0$ and $b_\tau = b(\tau, \delta)$ so that for all $n\geq m$ and for any $\omega \in \mathcal{H}$,
	\begin{equation}\label{eq:Intervalinequality}
		\int_I V_{\delta}^{(\tau)}(g_{n\tau} r_\theta \omega) \,d\theta < c e^{-(1-\delta)n\tau} V_{\delta}^{(\tau)}(\omega) + b_\tau |I|.
	\end{equation} 
\end{prop}
\begin{proof}
	Let $n\geq m$ and $\omega \in \mathcal{H}$. Our goal is to construct the constants $c$ and $b_\tau$ so that Equation~\ref{eq:Intervalinequality} holds. Indeed applying Lemma~\ref{lem:DozierShadowing} followed by Lemma~\ref{lem:JayadevsThesis}, we have 
	\begin{align*}
		\int_{I} V_{\delta}^{(\tau)} (g_{n\tau}r_\theta \omega) \,d\theta &\leq c_2 \int_{J_{n-1}} \text{Ave}_{\tau}(V_{\delta}^{(\tau)})(g_{(n-1)\tau} r_\theta \omega) \,d\theta \\
		&\leq c_2 \left(\int_{J_{n-1}} c_1 e^{-(1-\delta)\tau} V_{\delta}^{(\tau)}(g_{(n-1)\tau} r_\theta \omega)\,d\theta + b_\tau\right)\\
		&= c_2c_1 e^{-(1-\delta) \tau} \int_{J_{n-1}} V_{\delta}^{(\tau)}(g_{(n-1)\tau}r_\theta\omega)\,d\theta + c_2 b_\tau |I|
	\end{align*}
	where the last equality follows from the fact that $|J_{n-1}| = |I|$.
	
	Now repeatedly applying this inequality for $n-1, n-2,\ldots, m$ with $I$ replaced by $J_{n-1}, J_{n-2}$ through $J_{m}$ which all have length $|I|$, we obtain
	\begin{align*}
		\int_{I} V_{\delta}^{(\tau)}(g_{n\tau}r_\theta \omega) &\leq \left(c_1c_2 e^{-(1-\delta)\tau}\right)^{n-m + 1}\int_{J_{m-1}} V_{\delta}^{(\tau)} (g_{(m-1)\tau} r_\theta \omega)\,d\theta \\
		&\quad+ |I| b_{\tau}c_2 \sum_{j=0}^{n-m+1} \left(c_1c_2e^{-(1-\delta)\tau}\right)^j\\
		&\leq  \left(c_1c_2 e^{-(1-\delta)\tau}\right)^{n-m + 1}\int_{J_{m-1}} V_{\delta}^{(\tau)} (g_{(m-1)\tau} r_\theta \omega)\,d\theta +  |I| b_\tau, \\
		\end{align*}
		where in the last inequality we replaced $b_\tau$ with $2c_2b_\tau$.
		By the logsmooth property of $V_{\delta}$ from Lemma~\ref{lem:JayadevsThesis}, splitting into small steps, there exists some $k(m,\tau)$ so that 
		$$V_{\delta}^{(\tau)}(g_{(m-1)\tau} r_\theta \omega) \leq c_3^{k(m,\tau)} V_{\delta}^{(\tau)}(\omega).$$
		
		Thus we can write our constant $c$ as
		$$c = (c_1c_2)^{n-m+1} \left(e^{-(1-\delta)\tau}\right)^{-m+1} c_3^{k(m,\tau)} |I| $$
		where we note $m$ depends on $\tau$ and $|I|$, so $c$ depends only on $\delta$, $\tau$ and $|I|$. \qed \end{proof}

\begin{coro}\label{cor:AverageUpperBound}
	Fix a connected component of a  stratum $\mathcal{H}$ and $0<\delta< 1$. There exists $\tau > 0$ so that for any interval $I\subseteq S^1$, there exists constants $c= c(\tau, \delta, |I|)  >0$ and $b_\tau = b(\tau,\delta)$ so that there exists an $\ell_0>0$ so that for all $\ell \geq \ell_0$ and for any $\omega \in \mathcal{H}$, 
	$$\int_I V_{\delta}^{(\tau)}(g_\ell r_\theta \omega) \,d\theta \leq ce^{-(1-\delta)\ell}V_{\delta}^{(\tau)}(\omega) + b_\tau |I|.$$
\end{coro}
\begin{proof}	
	We choose $\tau$ to satisfy the assumption of Proposition~\ref{prop:VdeltaAverageInterval}, and for $\lambda>0$, set $\mathcal{V}$ in the $KAK$ decomposition to be a left $K$-invariant sets which contains $g_t$ for all $-\lambda < t< \lambda$. Choose $\ell_0 = m\tau >0$ where $m$ is defined in Proposition~\ref{prop:VdeltaAverageInterval}. Let $\ell \geq \ell_0$. Pick $n_0 = \min\{n\in \N: n\tau > \ell\}$
	and note $n_0 - 1 \geq m$. Let $r = n_0\tau - \ell$. Set $k = \lceil \frac{r}{\lambda}\rceil$, and $r_0 = \frac{r}{k}$. Then $r_0 \leq \lambda$, which implies ${g_{r_0}} \in \mathcal{V}$ and we can apply \eqref{eq:logsmooth}. We also want a lower bound on $r_0$. Assume $r_0 < \frac{\lambda}{3}$, in which case $k > 3r/\lambda \geq 3(k-1)$ and we conclude $k\leq 3/2$ and thus $k=1$. Thus if $k\geq 2$ we have $r_0 \geq \frac{\lambda}{3}$, or $k=1$. 	
	Then from Proposition~\ref{prop:VdeltaAverageInterval},
	\begin{align*}
		\int_I V_{\delta}^{(\tau)} (g_{\ell} r_\theta \omega)\,d\theta &= \int_I V_{\delta}^{(\tau)}(g_{-kr_0}g_{n_0\tau} r_\theta \omega)\,d\theta\\
		&\leq c_3^k \int_I V_{\delta}^{(\tau)}(g_{n_0\tau} r_\theta \omega) \,d\theta \tag{by \eqref{eq:logsmooth}}\\
		&\leq c_3^k \left[ce^{-(1-\delta) n_0\tau} V_{\delta}^{(\tau)}(\omega) + b_\tau |I|\right] \tag{by \eqref{eq:Intervalinequality}}\\
		&\leq c_3^k \left[ce^{-(1-\delta) \ell} V_{\delta}^{(\tau)}(\omega) + b_\tau |I|\right] \tag{since $n_0\tau \geq \ell$}
		\end{align*}
		Thus if $k=1$, we obtain the final constant which is independent of $\ell$. Otherwise we assume $k\geq 2$, and 
		\begin{align*}
		\int_I V_{\delta}^{(\tau)} (g_{\ell} r_\theta \omega)\,d\theta&\leq c_3^{\frac{\tau}{r_0}} \left[ce^{-(1-\delta) \ell} V_{\delta}^{(\tau)}(\omega) + b_\tau |I|\right] \tag{since $r \leq \tau$}\\
		&\leq c_3^{\frac{3\tau}{\lambda}} \left[ce^{-(1-\delta) \ell} V_{\delta}^{(\tau)}(\omega) + b_\tau |I|\right].
	\end{align*}
	Thus  given the lower bound on $r_0 \geq \frac{\lambda}{3}$, the final constants only depend on $\tau, \lambda$ and not $\ell$ and we obtain the desired result.
\qed \end{proof}	

\subsection{Completion of the verification of (4) (d)} \label{subsec:4d}
To obtain an upper bound for $\mu(B_i\cap C_j)$, by $g_t$-invariance of $\mu$, it suffices to find an upper bound for $\mu(\tilde{B}_i \cap \tilde{C}_j)$ where $\tilde{B}_i = \bigcup_{\theta \in I} r_\theta W_i= I\cdot W_i$ and $\tilde{C}_j = g_{f(i,j)}S(c_{\mathcal{H}},\sigma, 2^j)$ 
		for $f(i,j) = \log\left(b^{j-i} \sqrt{\frac{\psi(b^j)}{\psi(b^i)}}\right).$
		
%{\color{blue}		We first use the fact that in $S(c_{\mathcal{H}},\sigma, 2^{j})$, the shortest possible saddle connection has length $\ell(\omega) \in \left(c_{\mathcal{H}}\sqrt{\frac{\sigma}{\psi(b^{j})}}, \sqrt{\frac{2\sigma}{\psi(b^{j})}}\right).$ Choose $\tau$ large enough to satisfy the assumption of Corollary~\ref{cor:AverageUpperBound}. By \eqref{eq:v_deltacomparabletoalpha1}, 
%		$$C_{\delta,\tau}^{-1} \left(\frac{\psi(b^{j})}{2\sigma}\right)^{\frac{1+\delta}{2}} \leq V_\delta^{(\tau)}(\omega) \leq \frac{C_{\delta,\tau}}{c_{\mathcal{H}}}\left(\frac{\psi(b^{j})}{\sigma}\right)^\frac{1+\delta}{2}.$$
		
%Thus by Markov's inequality,}

{
Note,  the shortest saddle connection on $\omega \in S(c_{\mathcal{H}},\sigma, b^{j})$ is at most 
		$\sqrt{\frac{2\sigma}{\psi(b^j)} }$ and so by (4.2) for such an $\omega$ we have $$C_{\delta,\tau}^{-1} \left(\frac{\psi(b^{j})}{2\sigma}\right)^{\frac{1+\delta}{2}} \leq V_\delta^{(\tau)}(\omega).$$ So, if $\omega'\in \tilde{C}_j$, we have 
		$V_{\delta}^{(\tau)}(g_{-f(i,j)}\omega') \geq C_{\delta,\tau}^{-1} \left(\frac{\psi(b^{j})}{2\sigma}\right)^{\frac{1+\delta}{2}} .$}
Thus, %by Markov's inequality, 
{because $\mu(\{x \in S: h(x)>C\})\leq C^{-1}\int_S |h(x)|\,d\mu(x)$} {we have}
		\begin{align} \label{eq:align1}
			\mu(\tilde{B}_i \cap \tilde{C}_j) &\leq \mu\left(\left\{\omega\in \tilde{B}_i: V_\delta^{(\tau)}(g_{-f(i,j)}\omega) \geq C_{\delta,\tau}^{-1} \left(\frac{\psi(b^{j})}{2\sigma}\right)^\frac{1+\delta}{2}\right\}\right)\\
			&\leq C_{\delta,\tau}\left(\frac{\psi(b^j)}{2\sigma}\right)^{-\frac{1+\delta}{2}} \int_{I\cdot W_i} V_{\delta}^{(\tau)} ( g_{-f(i,j)} \omega) \,d\mu(\omega). \nonumber
\end{align}
Disintegrating the measure $\mu = \,d\theta \,d \tilde{\mu}$ on $\mathrm{SO}(2) \times (\mathcal{H} / \mathrm{SO}(2) )$ and increasing to a full $\mathrm{SO}(2)$ orbit
\begin{align*}
			\eqref{eq:align1} &\leq C_{\delta,\tau}\left(\frac{\psi(b^j)}{2\sigma}\right)^{-\frac{1+\delta}{2}} \int_{\mathrm{SO}(2) \cdot W_i / \mathrm{SO}(2)} \int_I V_{\delta}^{(\tau)} (g_{-f(i,j)}  r_\theta \tilde{\omega}) \,d\theta \,d\tilde{\mu}(\tilde{\omega}) \\
			&\leq C_{\delta,\tau}\left(\frac{\psi(b^j)}{2\sigma}\right)^{-\frac{1+\delta}{2}} \int_{\mathrm{SO}(2) \cdot W_i / \mathrm{SO}(2)} ce^{-(1-\delta)(f(i,j))} V_\delta^{(\tau)}(\tilde{\omega}) + b_\tau |I| \,d\tilde{\mu}(\tilde{\omega}) \tag{by Corollary~\ref{cor:AverageUpperBound} and monotonicity of $\psi$, $f(i,j) \geq \log(b^{j-i}) = \ell_0(j-i)> \ell_0$}\\
			&= C_{\delta,\tau}\left(\frac{\psi(b^j)}{2\sigma}\right)^{-\frac{1+\delta}{2}} \left(ce^{-(1-\delta)(f(i,j))}\int_{W_i} V_\delta^{(\tau)}(\omega)\, d\mu(\omega) + b_\tau |I| \mu(W_i)\right). \tag{since $\mathrm{SO}(2)W_i$ is $\mathrm{SO}(2)$-invariant}\\
\end{align*}
Since all holonomy vectors in $W_i$ are contained in the circle we apply the upper bound for $V_\delta^{(\tau)}$
\begin{align*}
			\eqref{eq:align1}&\leq C_{\delta,\tau}\left(\frac{\psi(b^j)}{2\sigma}\right)^{-\frac{1+\delta}{2}} \mu(W_i) \left(ce^{-(1-\delta)(f(i,j))} \frac{C_{\delta,\tau}}{c_{\mathcal{H}}} \left(\frac{\psi(b^{j})}{\sigma}\right)^\frac{1+\delta}{2} + b_\tau |I|\right) \\
			&\leq \mu(\tilde{B}_i) \left({\frac{c}{c_\mathcal{H}}e^{-(1-\delta)f(i,j)} C_{\delta,\tau}^2 2^{\frac{1+\delta}{2}}} + b_\tau |I| C_{\delta,\tau}\left(\frac{2\sigma}{\psi(b^j)}\right)^\frac{1+\delta}{2} \right).\tag{since $W_i \subseteq \tilde{B}_i$.} 
		\end{align*}
		
		Our goal is to compare the equation on the right hand side to the volume of $C_j$. Note by the construction of the MSV measure, there is a constant $m$ so that 
		$$\mu(C_j) = \mu(S(c_{\mathcal{H}},\sigma, 2^j)) \geq m \left(\sqrt{\frac{2\sigma}{\psi(b^j)}}\right)^2 = m \frac{2\sigma}{\psi(b^j)}.$$
		
		Thus we have 
		\begin{align*}
			&\mu(B_j\cap C_j) \\
			&= \mu(\tilde{B}_j \cap \tilde{C}_j) \\
			&\leq \mu(B_i) \left({ \frac{c}{c_{\mathcal{H}}}} C_{\delta, \tau}^2 2^{\frac{1+\delta}{2}} e^{-(1-\delta)f(i,j)} + b_\tau |I| C_{\delta, \tau} \left(\frac{\mu(C_j)}{m}\right)^\frac{1+\delta}{2} \right)\\
			&= \mu(B_i) \left({ \frac{c}{c_{\mathcal{H}}}}C_{\delta, \tau}^22^{\frac{1+\delta}{2}} 2^{-(j-i)(1-\delta)} \left(\frac{\psi(b^i)}{\psi(b^j)}\right)^\frac{1-\delta}{2} + b_\tau |I| C_{\delta, \tau} \left(\frac{\mu(C_j)}{m}\right)^\frac{1+\delta}{2} \right) \tag{by the definition of $f(i,j)$}\\
			&\leq \mu(B_i) \left({ \frac{c}{c_{\mathcal{H}}}}C_{\delta, \tau}^22^{\frac{1+\delta}{2}} 2^{-(j-i)(1-\delta)} + \frac{b_\tau |I| C_{\delta, \tau}}{m^\frac{1+\delta}{2}} \mu(C_j)^\frac{1+\delta}{2} \right)\tag{since $\psi(R)$ is a nondecreasing sequence and $i<j$, $\psi(b^i) \leq \psi(b^j)$}\\
			&\leq \mu(B_i) \left(c_1 2^{-(j-i)(1-\delta)} + c_2 \mu(C_j)^{\frac{1+\delta}{2}}\right). \tag{defining $c_1 = { \frac{c}{c_{\mathcal{H}}}}C_{\delta, \tau}^22^{\frac{1+\delta}{2}}$ and $c_2 = \frac{b_\tau |I| C_{\delta, \tau} }{m^\frac{1+\delta}{2}}$}
		\end{align*}
		
		Picking $C > \max\{c_1, c_2\}$ we obtain the desired inequality.

\begin{acknowledgements}
We would like to thank Jayadev Athreya, Osama Khalil, and Howard Masur for useful discussions. We also thank the anonymous referee for improving the readability of the paper.
\end{acknowledgements}

% Authors must disclose all relationships or interests that 
% could have direct or potential influence or impart bias on 
% the work: 
%
 \section*{Conflict of interest}
 The authors declare that they have no conflict of interest.

% BibTeX users please use one of
%\bibliographystyle{spbasic}      % basic style, author-year citations
\bibliographystyle{spmpsci}      % mathematics and physical sciences
%\bibliographystyle{spphys}  
     % APS-like style for physics
\bibliography{Sources}   % name your BibTeX data base

% Non-BibTeX users please use
%\begin{thebibliography}{}
%
% and use \bibitem to create references. Consult the Instructions
% for authors for reference list style.
%
%\bibitem{RefJ}
% Format for Journal Reference
%Author, Article title, Journal, Volume, page numbers (year)
% Format for books
%\bibitem{RefB}
%Author, Book title, page numbers. Publisher, place (year)
% etc
%\end{thebibliography}

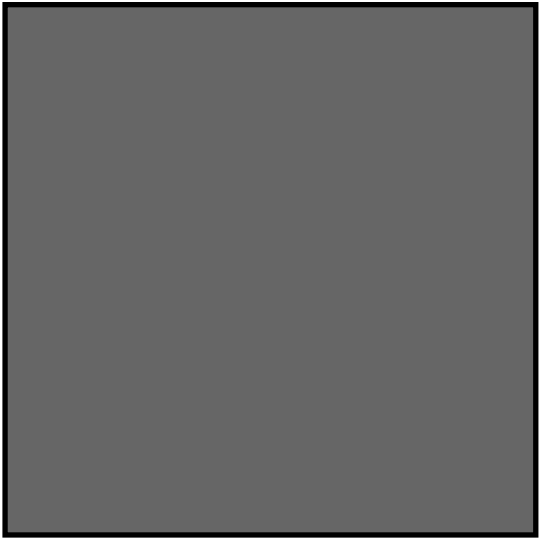
\end{document}